\documentclass[12pt]{amsart}
\usepackage[letterpaper,margin=1.1in]{geometry}
\usepackage{graphicx}
\usepackage{amssymb}
\usepackage[nobysame]{amsrefs}

\newtheorem{theorem}{Theorem}[section]
\newtheorem{lemma}[theorem]{Lemma}
\newtheorem{corollary}[theorem]{Corollary}
\newtheorem{proposition}[theorem]{Proposition}

\theoremstyle{definition}
\newtheorem{definition}[theorem]{Definition}
\newtheorem{remark}[theorem]{Remark}
\newtheorem{example}[theorem]{Example}
\newtheorem{problem}[theorem]{Problem}
\newtheorem*{ack}{Acknowledgement}

\theoremstyle{remark}

\newcommand{\RR}{\mathbb{R}}

\newcommand{\BMO}{\mathrm{BMO}}
\newcommand{\VMO}{\mathrm{VMO}}
\newcommand{\dist}{\mathop\mathrm{dist}}
\newcommand{\interior}{\mathop\mathrm{int}}

\newcommand{\surf}{\mathcal{H}^{n-1}}
\newcommand{\HD}{\mathop\mathrm{HD}}
\newcommand{\Good}{\mathrm{Good}}
\newcommand{\Null}{\mathrm{Null}}

\def\res{\hbox{ {\vrule height .22cm}{\leaders\hrule\hskip.2cm} } }

\numberwithin{equation}{section}
\numberwithin{figure}{section}

\begin{document}

\title[Harmonic Polynomials and Harmonic Measure]{Flat Points in Zero Sets of Harmonic Polynomials and Harmonic Measure from Two Sides}
\author{Matthew Badger}
\thanks{The author was partially supported by NSF grant DMS-0838212.}
\date{April 26, 2012}
\address{Department of Mathematics\\ Stony Brook University\\ Stony Brook, NY 11794-3651}
\email{badger@math.sunysb.edu}

\begin{abstract} We obtain quantitative estimates of local flatness of zero sets of harmonic polynomials. There are two alternatives: at every point either the zero set stays uniformly far away from a hyperplane in the Hausdorff distance at all scales or the zero set becomes locally flat on small scales with arbitrarily small constant. An application is given to a free boundary problem for harmonic measure from two sides, where blow-ups of the boundary are zero sets of harmonic polynomials.
\end{abstract}

\maketitle

\section{Introduction} % use lowercase except for proper names
\label{intro}

In this paper, we study a geometric property of the zero sets of harmonic polynomials in order to gain new information about free boundary regularity for harmonic measure from two sides. To briefly describe this application, assume that $\Omega^+=\Omega$ and $\Omega^-=\RR^n\setminus\overline{\Omega}$ are complementary domains with a common boundary $\partial\Omega^+=\partial\Omega=\partial\Omega^-$. Roughly speaking, we wish to know what does the boundary look like, when the harmonic measure from one side of the boundary and the harmonic measure from the opposite side of the boundary look the same. Thus assume that the harmonic measures $\omega^\pm$ of $\Omega^\pm$ charge the same sets (i.e.\ $\omega^+(E)=0 \Leftrightarrow \omega^-(E)=0$ for all Borel sets $E\subset\partial\Omega$). In Badger \cite{Badger1} (refining previous work by Kenig and Toro \cite{KT06}) the author established the following structure theorem for the free boundary under weak regularity:
\begin{quotation}\noindent If the Radon-Nikodym derivative $f=d\omega^-/d\omega^+$ has continuous logarithm as a function on $\partial\Omega$, then the boundary decomposes as a finite disjoint union of sets $\Gamma_d$ ($1\leq d\leq d_0$), \[ \partial\Omega=\Gamma_1\cup\dots\cup\Gamma_{d_0}, \] with the following property. Every blow-up of $\partial\Omega$ centered a point $x\in\Gamma_d$ (i.e.\ a limit of the sets $r_i^{-1}(\partial\Omega-x)$ as $r_i\rightarrow 0$ in a Hausdorff distance sense) is the zero set of a homogeneous harmonic polynomial of degree $d$.\end{quotation}
(For a precise formulation of the structure theorem, see section \ref{SectFreeBoundary} below.)
In other words, ``zooming in" on a point in the boundary, the limiting shapes that one sees are zero sets of homogeneous harmonic polynomials. Moreover, the degrees of the polynomials which appear in this fashion are uniquely determined at each point of the boundary. In particular, every boundary point belongs either to the set of ``flat points" $\Gamma_1$ where blow-ups of the boundary are hyperplanes, or to the set of ``singularities" $\Gamma_2\cup\dots\cup\Gamma_{d_0}$ where blow-ups of the boundary are zero sets of higher degree homogeneous harmonic polynomials (see Figure 1.1).
Below we study the topology, geometry, and size of the set of flat points $\Gamma_1$. We will show that $\Gamma_1$ is open in $\partial\Omega$, $\Gamma_1$ is locally Reifenberg flat with vanishing constant, and thus, $\Gamma_1$ has Hausdorff dimension $n-1$.

\begin{figure}[t]
\includegraphics[width=.3\textwidth]{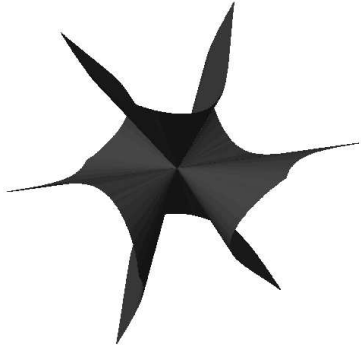}
\caption{A blow-up of $\partial\Omega$, $\Omega\subset\RR^3$ about $x\in\Gamma_3$ is the zero set of a homogeneous harmonic polynomial of degree 3 such as $x^2(y-z)+y^2(z-x)+z^2(x-y)-xyz$}
\end{figure}

The main tool that we need to study $\Gamma_1$ may be of independent interest. It is a statement about the local geometry of zero sets of harmonic polynomials, which connects analytic and geometric notions of ``regular points". While analytic regularity of a zero set at a point is indicated by the non-vanishing of the Jacobian of a defining function, geometric regularity of a zero set at a point is displayed by the existence of a tangent plane to the set. Alternatively, one may equate geometric regularity with existence of arbitrarily good approximations of the set by hyperplanes at small scales. We will show that for zero sets of harmonic polynomials these two types of regularity---analytic and geometric---coincide. Moreover, we quantify the failure of the zero set to admit good approximations by hyperplanes at its singularities. In order to state our result precisely, we need to introduce some notation.

Let $\Sigma\subset\RR^n$ $(n\geq 2$) be a closed set. The \emph{local flatness} $\theta_\Sigma(x,r)$ of $\Sigma$ near the point $x\in \Sigma$ and at scale $r>0$ is defined by (c.f. measurements of flatness in \cite{Jones}, \cite{MV}, \cite{Toro}) \begin{equation}\label{thetaapprox}\theta_\Sigma(x,r)=\frac{1}{r}\min_{L\in G(n,n-1)} \HD[\Sigma\cap B(x,r), (x+L)\cap B(x,r)],\end{equation} where as usual $G(n,n-1)$ denotes the collection of $(n-1)$-dimensional subspaces of $\RR^n$ (hyperplanes through the origin) and $\HD[A,B]$ denotes the Hausdorff distance between nonempty, compact subsets of $\RR^n$, \begin{equation}\HD[A,B]=\max\left\{\sup_{x\in A}\,\dist(x,B),\ \sup_{y\in B}\, \dist(y,A)\right\}.\end{equation} Thus local flatness is a measure of how well a set can be approximated by a hyperplane at a given location and scale (see Figure 1.2). Notice that $\theta_\Sigma(x,r)$ measures the distance of points in the set to a hyperplane \emph{and} the distance of points in a hyperplane to the set. The minimum in (\ref{thetaapprox}) is achieved for some hyperplane $L_{x,r}$ by compactness of $G(n,n-1)$; however, for typical sets $L_{x,r}$ may vary with $r$. Because the local flatness $\theta_\Sigma(x,r)\leq 1$ for every closed set $\Sigma$, for every $x\in\Sigma$ and for every $r>0$ this quantity only carries information when $\theta_\Sigma(x,r)$ is small.

\begin{figure}[t]\label{thetaflat}\begin{center}\includegraphics[width=.375\textwidth]{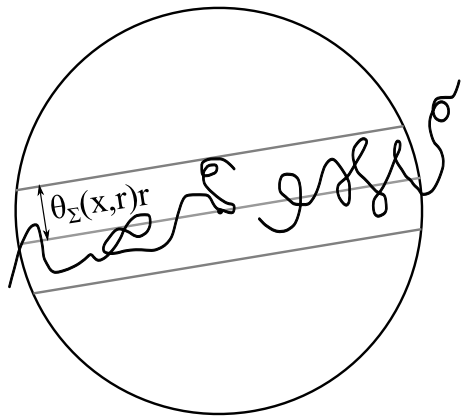}\end{center}\caption{Local flatness $\theta_\Sigma(x,r)$ of a set $\Sigma$ at scale $r$}\end{figure}

If $x\in\Sigma$ and $\lim_{r\rightarrow 0} \theta_\Sigma(x,r)=0$, then we say that $x$ is a \emph{flat point} of $\Sigma$. From our viewpoint, flat points are the ``geometric regular" points of $\Sigma$. Let us give two examples with zero sets. Given any $f:\RR^n\rightarrow\RR$, we write $\Sigma_f=\{x\in\RR^n:f(x)=0\}$ for the zero set of $f$ and we write $Df$ for the total derivative of $f$.

\begin{example} Suppose that $f:\RR^n\rightarrow\RR$ is smooth. If $x\in\Sigma_f$ and $Df(x)\neq 0$, then $\Sigma_f$ admits a unique tangent plane at $x$. Thus $x\in\Sigma_f$ is a flat point of $\Sigma_f$ whenever $Df(x)\neq 0$.\end{example}

\begin{example}[Tacnode] \label{egtacnode} The polynomial $p(x,y)=x^4+y^4-y^2$ has a singularity at the origin, i.e.\  $Dp(0)=0$. Nevertheless the origin is a flat point of $\Sigma_p$ (see Figure 1.3).
\end{example}

\begin{figure}[t]\label{tacnode}\begin{center}\includegraphics[width=.45\textwidth]{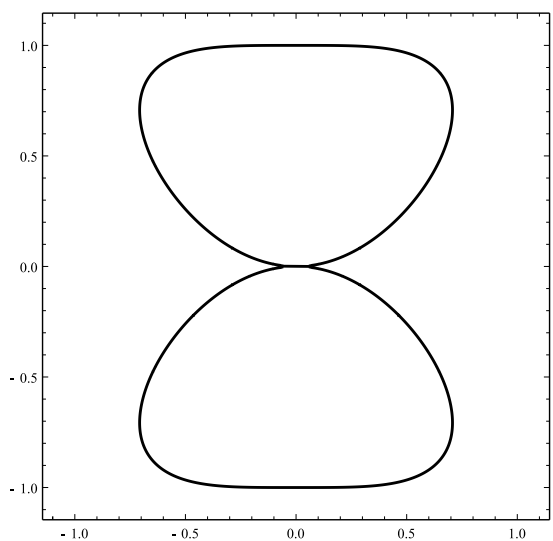}\end{center}\caption{Zero set of $x^4+y^4-y^2$}\end{figure}

These examples show that while every (analytic) regular point of the zero set of a smooth function is a flat point, the converse does not hold for a general smooth function. However, as we shall see, the converse does hold for zero sets of harmonic polynomials. Our convention below is that $p$ denotes a generic polynomial and $h$ denotes a harmonic polynomial.

\begin{theorem}\label{harmonicflat} If $h:\RR^n\rightarrow\RR$ is a nonconstant harmonic polynomial, then every flat point of $\Sigma_h$ is a regular point of $\Sigma_h$:  $\{x\in\Sigma_h:x$ is flat point of $\Sigma_h\}=\{x\in\Sigma_h:Dh(x)\neq 0\}$.\end{theorem}

In fact, we establish the following stronger, quantitative statement. It says that if the zero set of a harmonic polynomial is sufficiently close to a hyperplane at a single scale, then one can automatically conclude the set is flat at that location.

\begin{theorem}\label{polythm}For all $n\geq 2$ and $d\geq 1$ there exists a constant $\delta_{n,d}>0$ such that for any harmonic polynomial $h:\RR^n\rightarrow\RR$ of degree $d$ and for any $x\in\Sigma_h$,
\begin{eqnarray*} Dh(x)=0\quad&\Leftrightarrow\quad&\theta_{\Sigma_h}(x,r)\geq \delta_{n,d}\quad\hbox{for all }r>0,\\
 Dh(x)\neq 0\quad&\Leftrightarrow\quad&\theta_{\Sigma_h}(x,r)< \delta_{n,d}\quad\hbox{for some }r>0.\end{eqnarray*} Moreover, there exists a constant $C_{n,d}>1$ such that if $\theta_{\Sigma_h}(x,r)<\delta_{n,d}$ for some $r>0$, then $\theta_{\Sigma_h}(x,sr)< C_{n,d} s$ for all $s\in(0,1)$.
\end{theorem}

The tacnode at the origin in Example \ref{egtacnode} shows that in general the zero set of a polynomial can be flat at a singularity of the polynomial. In contrast, Theorem \ref{polythm} says that for any harmonic polynomial $h:\RR^n\rightarrow\RR$ the zero set $\Sigma_h$ is ``far away from flat" (i.e.\ $\theta_{\Sigma_h}(x,r)\geq \delta_{n,d}$ for all $r>0$) at every singularity of the polynomial, uniformly across all harmonic polynomials of a specified degree. Thus for harmonic polynomials the set of flat points of $\Sigma_h$ (geometric regularity) and regular points of $\Sigma_h$ (analytic regularity) coincide. It would be interesting to know for which functions this property holds.

\begin{problem} Classify all polynomials $p:\RR^n\rightarrow\RR$ such that the flat points of $\Sigma_p$ and the regular points of $\Sigma_p$ coincide, i.e. such that $x\in\Sigma_p$ is flat implies $Dp(x)\neq 0$.\end{problem}

\begin{problem} Classify all smooth functions $F:\RR^n\rightarrow\RR$ such that the flat points of $\Sigma_F$ and the regular points of $\Sigma_F$ coincide, i.e. such that $x\in\Sigma_F$ is flat implies $DF(x)\neq 0$.\end{problem}

To prove Theorem \ref{polythm} we identify a certain quantity $\zeta_1(p,x,r)\in[0,\infty]$ which measures the ``relative size of the linear term" of a polynomial $p:\RR^n\rightarrow\RR$. This quantity depends continuously on the coefficients of $p$, identifies whether $Dp(x)$ vanishes, and bounds the local flatness $\theta_{\Sigma_p}(x,r)$ of $\Sigma_p$ from above. Moreover, at any $x\in\Sigma_p$, the quantity $\zeta_1(p,x,r)$ decays linearly, in the sense that $\zeta_1(p,x,sr)\leq s\zeta_1(p,x,r)$ for all $s\in(0,1)$. To establish Theorem \ref{polythm}, the critical step is to show that for harmonic polynomials $\zeta_1(h,x,r)<\infty$ whenever $\theta_{\Sigma_h}(x,r)$ is sufficiently small (see Proposition \ref{converselemma}). The key facts about harmonic polynomials which are useful for this purpose are the mean value property for harmonic functions and estimates on the Lipschitz constant of spherical harmonics from \cite{Badger1}.

The remainder of the paper is divided into two parts. In the first part, \S\S \ref{SectRelativeSizes}--\ref{SectFlatnessZeroSets}, we build up the proof of Theorem 1.4. To start, in \S\ref{SectRelativeSizes} we define the relative size $\zeta_k(p,x,r)$ of the homogeneous part of degree $k$, of a polynomial $p$, on the ball $B(x,r)$. Then we record the basic properties of these numbers, which are used in the sequel. Section \ref{SectConvergenceZeroSets} proceeds with a brief discussion on convergence of zero sets of polynomials in the Hausdorff distance. In particular, in Corollary \ref{hptohomog}, we identify the blow-ups in the Hausdorff distance sense of zero sets of harmonic polynomials. Section \ref{SectFlatnessZeroSets} is devoted to the connection between the relative size of the linear term $\zeta_1(p,x,r)$ and the local flatness $\theta_{\Sigma_p}(x,r)$ of the zero set. First we show that for any polynomial, not necessarily harmonic, the relative size of the linear term controls local flatness of the zero set (Lemma \ref{zetaimpliesflat}). To establish a converse for harmonic polynomials, we first demonstrate that zero sets of homogeneous harmonic polynomials of degree $d\geq 2$ are uniformly far away from flat at the origin (Lemma \ref{hlemma}). We then pass to a converse for zero sets of generic harmonic polynomials and the proof of Theorem \ref{polythm}, using a normal families/blow-up type argument and the technology of \S\ref{SectRelativeSizes} and \S\ref{SectConvergenceZeroSets}.

In the second part, \S\S\ref{SectApproximation}--\ref{SectFreeBoundary}, we turn to applications of Theorem \ref{polythm}. In \S\ref{SectApproximation}, we examine a variant of Reifenberg flat sets, where local approximations of a set by hyperplanes at small scales are replaced with local approximations by zero sets of harmonic polynomials (see Definitions \ref{lapprox} and \ref{HdApprox}). Using Theorem \ref{polythm}, we deduce that if a set $\Sigma$ admits arbitrarily close local approximations by zero sets of harmonic polynomials, then the local flatness $\theta_{\Sigma}(x,r)$ of $\Sigma$ at one scale $r>0$ yields good control of the local flatness $\theta_{\Sigma}(x,r')$ of $\Sigma$ at all smaller scales $0<r'\leq r$ (see Lemma \ref{emu}). If, in addition, all blow-ups of $\Sigma$ are zero sets of \emph{homogeneous} harmonic polynomials, then the subset $\Sigma_1$ of flat points of $\Sigma$ is open; and $\Sigma_1$ is locally Reifenberg flat with vanishing constant (Theorem \ref{generalopenthm}). Finally, in \S\ref{SectFreeBoundary}, we specialize the results from \S\ref{SectApproximation} to the setting of free boundary regularity for harmonic measure from two sides discussed above. In particular, we obtain refined information about the set of flat points $\Gamma_1$ in the free boundary $\partial\Omega$. We end with a list of open problems about free boundary regularity for harmonic measure from two sides.

\section{Relative Size of Homogeneous Parts of a Polynomial}
\label{SectRelativeSizes}

Let $x\in\RR^n$. A polynomial $p:\RR^n\rightarrow\RR$ of degree $d\geq 1$ decomposes as \begin{equation}\label{hd0}p(z)=p^{(x)}_d(z-x)+\dots+p^{(x)}_1(z-x)+p^{(x)}_0(z-x)\end{equation} where each non-zero term $p^{(x)}_k$ is a \emph{homogenous} polynomial of degree $k$, i.e.\ \begin{equation}\label{hd1}p_k^{(x)}(ty)=t^kp_k^{(x)}(y)\quad\text{for all }t\in\RR\text{ and }y\in\RR^n.\end{equation} We call $p_k^{(x)}$ the \emph{homogeneous part} of $p$ of \emph{degree} $k$ with \emph{center} $x$. By Taylor's theorem, \begin{equation}\label{pkxdefn}p^{(x)}_k(y)=\sum_{|\alpha|=k}\frac{D^\alpha p(x)}{\alpha!}y^\alpha\quad\text{for all }y\in\RR^n.\end{equation} In the sequel, it will be convenient to quantify the relative sizes of homogeneous parts.

\begin{definition}\label{zetanums} Let $p:\RR^n\rightarrow\RR$ be a polynomial of degree $d\geq 1$. For every $0\leq k\leq d$, $x\in\RR^n$ and $r>0$, define \begin{equation}\label{zetadef}\zeta_k(p,x,r)=\max_{j\neq k}\frac{\|p^{(x)}_j\|_{L^\infty(B_r)}}{\|p^{(x)}_k\|_{L^\infty(B_r)}}\in[0,\infty].\end{equation}\end{definition}

\begin{remark} Definition \ref{zetanums} generalizes the two quantities $\zeta(h)$ and $\zeta_*(h)$ associated to a harmonic polynomial $h$, which appeared in Badger \cite{Badger1} (see Lemma 4.3 and Lemma 4.5). In the present notation, if $h=h^{(0)}_d+h^{(0)}_{d-1}+\dots+h^{(0)}_j$ is a harmonic polynomial of degree $d\geq 1$ such that $h(0)=0$ and $h^{(0)}_j\neq 0$, then $\zeta(h)=\zeta_d(h,0,1)$ and $\zeta_*(h)=\zeta_j(h,0,1)$.\end{remark}

Because $\zeta_k(p,x,r)$ measures the \emph{relative} size of homogeneous parts of a polynomial, scaling $p$ does not affect $\zeta_k$. This simple observation will enable proofs via normal families (for example, see the proof of Proposition \ref{converselemma}), by allowing us to assume a sequence of polynomials with certain properties has uniformly bounded coefficients.

\begin{lemma}\label{rescale} If $p:\RR^n\rightarrow\RR$ is a polynomial of degree $d\geq 1$ and $c\in\RR\setminus\{0\}$, then $\zeta_k(cp,x,r)=\zeta_k(p,x,r)$ for all $0\leq k\leq d$, $x\in\RR^n$ and $r>0$.\end{lemma} \begin{proof} Suppose that $p:\RR^n\rightarrow\RR$ is a polynomial of degree $d\geq 1$, and let $c\in\RR\setminus\{0\}$. Since $(cp)_k^{(x)}=c(p_k^{(x)})$ for all $0\leq k\leq d$, \begin{equation}\begin{split}\zeta_k(cp,x,r)=\max_{j\neq k}\frac{\|cp^{(x)}_j\|_{L^\infty(B_r)}}{\|cp^{(x)}_k\|_{L^\infty(B_r)}}
&=\max_{j\neq k}\frac{|c|\cdot\|p^{(x)}_j\|_{L^\infty(B_r)}}{|c|\cdot\|p^{(x)}_k\|_{L^\infty(B_r)}}\\ &=\max_{j\neq k}\frac{\|p^{(x)}_j\|_{L^\infty(B_r)}}{\|p^{(x)}_k\|_{L^\infty(B_r)}}=\zeta_k(p,x,r)\end{split}\end{equation} for all $x\in\RR^n$ and all $r>0$.\end{proof}

The quantity $\zeta_k(p,x,r)$ also behaves well under translation and dilation.

\begin{lemma}\label{recenter} Suppose that $p:\RR^n\rightarrow\RR$ is a polynomial of degree $d\geq 1$. If $z\in\RR^n$, then $\zeta_k(p(\cdot+z),x,r)=\zeta_k(p,x+z,r)$ for all $x\in\RR^n$, for all $r>0$ and for all $0\leq k\leq d$.\end{lemma}

\begin{proof} Let $p:\RR^n\rightarrow\RR$ be a polynomial of degree $d\geq 1$, fix $x\in\RR^n$ and define $q:\RR^n\rightarrow\RR$ by $q(y)=p(y+z)$ for all $y\in\RR^n$. Then $q$ is a polynomial of degree $d$. Moreover, for all $0\leq k\leq d$, \begin{equation} q^{(x)}_k(y)=\sum_{|\alpha|=k}\frac{D^\alpha q(x)}{\alpha!}y^\alpha=\sum_{|\alpha|=k}\frac{D^\alpha p(x+z)}{\alpha!}y^\alpha=p^{(x+z)}_k(y)\quad\text{for all }y\in\RR^n.\end{equation} Thus, $q^{(x)}_k=p^{(x+z)}_k$ for all $x\in\RR^n$ and for all $0\leq k\leq d$. It immediately follows that $\zeta_k(q,x,r)=\zeta_k(p,x+z,r)$ for all $0\leq k\leq d$, for all $x\in\RR^n$ and for all $r>0$.\end{proof}

\begin{lemma}\label{dilate} If $p:\RR^n\rightarrow\RR$ is a polynomial of degree $d\geq 1$ and $t>0$, then $\zeta_k(p(t\cdot),x,r)=\zeta_k(p,tx,tr)$ for all $x\in\RR^n$, for all $r>0$ and for all $0\leq k\leq d$.
\end{lemma}

\begin{proof} Let $p:\RR^n\rightarrow\RR$ be a polynomial of degree $d\geq 1$, fix $t>0$ and define $q:\RR^n\rightarrow\RR$ by $q(y)=p(ty)$ for all $y\in\RR^n$. Then $q$ is a polynomial of degree $d$ and for all $0\leq k\leq d$, \begin{equation} q^{(x)}_k(y)=\sum_{|\alpha|=k}\frac{D^\alpha q(x)}{\alpha!}y^\alpha=t^k\sum_{|\alpha|=k}\frac{D^\alpha p(tx)}{\alpha!}y^\alpha=t^kp^{(tx)}_k(y)\quad\text{for all }y\in\RR^n.\end{equation} Hence $q^{(x)}_k=t^kp^{(tx)}_k$ for all $0\leq k\leq d$. It follows that \begin{equation}\begin{split}\zeta_k(q,x,r)=\max_{j\neq k}\frac{\|q^{(x)}_j\|_{L^\infty (B_r)}}{\|q^{(x)}_k\|_{L^\infty (B_r)}}&=\max_{j\neq k}\frac{t^j\|p^{(tx)}_j\|_{L^\infty (B_r)}}{t^k\|p^{(tx)}_k\|_{L^\infty (B_r)}}\\&=\max_{j\neq k}\frac{\|p^{(tx)}_j\|_{L^\infty (B_{tr})}}{\|p^{(tx)}_k\|_{L^\infty (B_{tr})}}=\zeta_k(p,tx,tr)\end{split}\end{equation} for all $x\in\RR^n$, for all $r>0$ and for all $0\leq k\leq d$.\end{proof}

The magnitude of $\zeta_k(p,x,r)$ identifies homogeneous polynomials and the vanishing of homogeneous parts of polynomials. For example, $p(x)=0$ if and only if $\zeta_0(p,x,r)=\infty$, and $Dp(x)=0$ if and only if $\zeta_1(p,x,r)=\infty$.

\begin{lemma}\label{maglemma} If $p:\RR^n\rightarrow\RR$ is a polynomial of degree $d\geq 1$, $x\in\RR^n$ and $0\leq k\leq d$, then  \begin{enumerate}
\item $\zeta_k(p,x,r)=0$ for all $r>0$ if and only if $p_k^{(x)}=p(\cdot+x)$;
\item $\zeta_k(p,x,r)>0$ for all $r>0$ if and only if $p_k^{(x)}\neq p(\cdot+x)$;
\item $\zeta_k(p,x,r)<\infty$ for all $r>0$ if and only if $p_k^{(x)}\neq 0$; and,
\item $\zeta_k(p,x,r)=\infty$ for all $r>0$ if and only if $p_k^{(x)}=0$.
\end{enumerate}
\end{lemma}
\begin{proof} We leave this exercise in the definition of $\zeta_k(p,x,r)$ to the reader.\end{proof}

The value of $\zeta_k(p,x,r)$ depends continuously on the coefficients of the polynomial $p$. To make this statement precise, we first make a definition.

\begin{definition} A sequence of polynomials $(p^i)_{i=1}^\infty$ in $\RR^n$ converges \emph{in coefficients} to a polynomial $p$ in $\RR^n$ if $d=\max_i \deg p^i<\infty$ and $D^\alpha p^i(0)\rightarrow D^\alpha p(0)$ for every $|\alpha|\leq d$.\end{definition}

\begin{lemma}\label{zetacts} For every $k\geq 0$, $\zeta_k(p,x,r)$ is jointly continuous in $p$, $x$ and $r$. That is, \begin{equation} \zeta_k(p^i,x_i,r_i)\rightarrow \zeta_k(p,x,r)\end{equation} whenever $p^i\rightarrow p$ in coefficients, $x_i\rightarrow x\in\RR^n$ and $r_i\rightarrow r\in(0,\infty)$.\end{lemma}

\begin{proof} Let $(p^i)_{i=1}^\infty$ be a sequence of polynomials in $\RR^n$ such that $p^i\rightarrow p$ in coefficients to a nonconstant polynomial $p$ and let $d=\max_i\deg p_i<\infty$. There are two cases. If $k>d$, then $p^{i(x_i)}_k=0$ for all $i\geq 1$ and $p_k^{(x)}=0$. Hence \begin{equation}\zeta(p^i,x_i,r_i)=\zeta(p,x,r)=\infty\quad\text{for all }i\geq 1\end{equation} by Lemma \ref{maglemma}. Otherwise $0\leq k\leq d$. Since \begin{equation}D^\alpha p(x)=\sum_{|\beta|\leq d-|\alpha|}\frac{D^{\alpha+\beta}p(0)}{\beta!}x^{\beta}\quad\text{for all }x\in\RR^n\text{ and }|\alpha|\leq d,\end{equation} convergence in coefficients implies that $D^\alpha p^i(x)\rightarrow D^\alpha p(x)$ for all $x\in\RR^n$ and $|\alpha|\leq d$, uniformly on compact subsets of $\RR^n$. From (\ref{pkxdefn}) it follows that $p_k^{i(x_i)}\rightarrow p_k^{(x)}$ uniformly on compact sets  whenever $p^i\rightarrow p$ in coefficients and $x_i\rightarrow x\in\RR^n$. Thus, for every $0\leq k\leq d$, \begin{equation}\|{p}_k^{i(x_i)}\|_{L^\infty(B_{r_i})}\rightarrow \|p_k^{(x)}\|_{L^\infty(B_{r})}\end{equation} whenever $p^i\rightarrow p$ in coefficients, $x_i\rightarrow x\in\RR^n$ and $r_i\rightarrow r\in (0,\infty)$.
We conclude that  \begin{equation}\label{lemon}
\max_{j\neq k}\frac{\|{p}_j^{i(x_i)}\|_{L^\infty(B_{r_i})}}{\|{p}_k^{i(x_i)}\|_{L^\infty(B_{r_i})}}\rightarrow \max_{j\neq k}\frac{\|{p}_j^{(x)}\|_{L^\infty(B_{r})}}{\|{p}_k^{(x)}\|_{L^\infty(B_{r})}}\in[0,\infty].\end{equation} (Note $0/0$ never appears in (\ref{lemon}) because the polynomials $p^i$ and $p$ are not identically zero.) That is, $\zeta_k(p^i,x_i,r_i)\rightarrow \zeta_k(p,x,r)$ whenever $p^i\rightarrow p$ in coefficients, $x_i\rightarrow x\in\RR^n$ and $r_i\rightarrow r\in(0,\infty)$, as desired.
\end{proof}

\begin{remark} If $(p^i)_{i=1}^\infty$ is a sequence of polynomials in $\RR^n$ such that $d=\max_i p^i<\infty$, then $p^i\rightarrow p$ in coefficients if and only if $p^i\rightarrow p$ uniformly on compact sets.\end{remark}

Next we show that the relative size of the linear term of a polynomial decays linearly at any root of the polynomial.

\begin{lemma}\label{lineardecay} If $p:\RR^n\rightarrow\RR$ is a polynomial of degree $d\geq 1$ and $p(x)=0$, then $\zeta_1(p,x,sr)\leq s\zeta_1(p,x,r)$ for all $r>0$ and $s\in(0,1)$.\end{lemma} \begin{proof} Suppose $p:\RR^n\rightarrow\RR$ is a polynomial of degree $d\geq 1$. First if $d=1$ and $p(x)=0$, then $p=p^{(x)}_1(\cdot-x)$ and $\zeta_1(p,x,r)=0$ for all $r>0$ by Lemma \ref{maglemma}. Second if $d\geq 2$ and $p(x)=0$, then $p=p^{(x)}_d(\cdot-x)+\dots+p^{(x)}_2(\cdot-x)+p^{(x)}_1(\cdot-x)$. Thus \begin{equation}\begin{split}\zeta_1(p,x,sr)=\max_{j>1}\frac{\|p^{(x)}_j\|_{L^\infty(B_{sr})}}{\|p^{(x)}_1\|_{L^\infty(B_{sr})}}
&=\max_{j>1}\frac{s^j\|p^{(x)}_j\|_{L^\infty(B_{r})}}{s\|p^{(x)}_1\|_{L^\infty(B_{r})}}\\ &\leq \max_{j>1}\frac{s^2\|p^{(x)}_j\|_{L^\infty(B_{r})}}{s\|p^{(x)}_1\|_{L^\infty(B_{r})}}=s\zeta_1(p,x,r)
\end{split}\end{equation} for all $r>0$ and $s\in(0,1)$.\end{proof}

Now let us specialize to harmonic polynomials.

\begin{lemma}\label{hparts} If $h$ is a harmonic polynomial in $\RR^n$ (i.e.\ $\Delta h=0$) of degree $d\geq 1$, then $h^{(x)}_k$ is harmonic for all $0\leq k\leq d$ and $x\in\RR^n$.\end{lemma}

\begin{proof}Suppose that $h$ is a harmonic polynomial of degree $d\geq 1$ and let $x\in\RR^n$. Applying Laplace's operator to (\ref{hd0}) yields \begin{equation}\label{laphd0}0=\Delta h^{(x)}_d + \Delta h^{(x)}_{d-1} + \dots + \Delta h^{(x)}_{2}.\end{equation} Since $\Delta h^{(x)}_k$ is the sum of monomials of degree $k-2$ for each non-zero $h^{(x)}_k$, the right hand side of (\ref{laphd0}) vanishes if and only if $\Delta h^{(x)}_k=0$ for all $0\leq k\leq d$. \end{proof}

\begin{remark} If $h:\RR^n\rightarrow\RR$ is any harmonic polynomial of degree $d\geq 1$, then \begin{equation}\zeta_k(h,x,r)=\max_{j\neq k}\frac{\|h^{(x)}_j\|_{L^\infty(\partial B_r)}}{\|h^{(x)}_k\|_{L^\infty(\partial B_r)}}\end{equation} by Lemma \ref{hparts} and the maximum principle for harmonic functions. \end{remark}

\section{Blow-ups and Convergence of Zero Sets of Polynomials}
\label{SectConvergenceZeroSets}

The following definition formalizes the notion of ``zooming in" on a closed set.

\begin{definition}\label{blowupdefn} Let $\Sigma\subset\RR^n$ be a nonempty closed set. A \emph{(geometric) blow-up $B$ of $\Sigma$ centered at $x\in \Sigma$} is a closed set $B\subset\RR^n$ such that for some sequence $r_i\downarrow 0$, \begin{equation} \lim_{i\rightarrow\infty}\HD\left[\frac{\Sigma-x}{r_i}\cap B_s,B\cap B_s\right]=0\quad\text{ for all } s>0\end{equation} where $B_s=B(0,s)$ denotes the closed ball at the origin with radius $s>0$.\end{definition}

The existence of blow-ups of a non-empty set is guaranteed by the following classical lemma. A proof may be found on page 91 of Rogers \cite{R}.

\begin{lemma}[Blaschke's selection theorem]\label{selectthm} Let $K\subset\RR^n$ be a compact set. If $(A_k)_{k=1}^\infty$ is a sequence of nonempty closed subsets of $K$, then there exists a nonempty closed set $A\subset K$ and a subsequence $(A_{k_j})_{j=1}^\infty$ of $(A_k)_{k=1}^\infty$ such that $\HD[A_{k_j},A]\rightarrow 0$ as $j\rightarrow\infty$.\end{lemma}

In this section, we shall identify the blow-ups of the zero set of a harmonic polynomial. But first we discuss the relationship between convergence of polynomials in coefficients and convergence of their zero sets in the Hausdorff distance. To start, we give an example which shows how and where issues can arise. Below $B(x,r)$ denotes the \emph{closed} ball with center $x\in\RR^n$ and radius $r>0$.

\begin{example} Let $h(x,y)=xy$ and for each $i\geq 1$ let $h^i(x,y)= h(x+1/i,y)=xy+y/i$. Then the polynomials $h$ and $h^i$ ($i\geq 1$) are harmonic and $h^i\rightarrow h$ in coefficients. However, we claim that there is a closed ball $B$ such that $\Sigma_{h^i}\cap B\neq\emptyset$ for all $i\geq 1$ and $\Sigma_h\cap B\neq\emptyset$ but $\Sigma_{h^i}\cap B$ does not converge to $\Sigma_h\cap B$ in the Hausdorff distance (see Figure 3.1). Indeed consider $B=B((1,1/2),1)$. Then $\Sigma_{h^i}\cap B=[1-c,1+c]\times \{0\}$, $c=\sqrt{3}/2$ is a fixed line segment for all $i\geq 1$, but $ \Sigma_h\cap B=([1-c,1+c]\times\{0\})\cup\{(0,1/2)\}$ consists of the line segment together with an additional point of $\partial B$. Thus convergence in coefficients does not imply (local) convergence of the zero sets in the Hausdorff distance in general. We remark that this extra point lies on $\partial B$ and is isolated in $\Sigma_h\cap B$, even though it is not isolated in $\Sigma_h$.
\end{example}

\begin{figure}[t] \begin{center}\includegraphics[width=.6\textwidth]{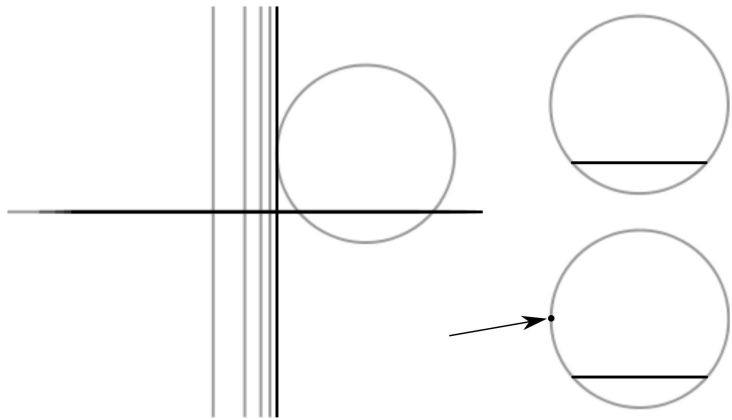}\end{center}\caption{Convergence in coefficients versus convergence of zero sets}\end{figure}

\begin{lemma}\label{hpfconverge} If $(h_i)_{i=1}^\infty$ is a sequence of harmonic polynomials and $h^i\rightarrow h$ in coefficients, then $h$ is a harmonic polynomial. Moreover, if $h$ is nonconstant and if $\Sigma_h\cap \interior B\neq\emptyset$ for some closed ball $B$, then $\Sigma_{h^i}\cap B\neq\emptyset$ for all $i\geq i_0$.  Furthermore, if $\Sigma_{h^i}\cap B$ converges in the Hausdorff distance to a closed set $F\subset B$, then $F\subset\Sigma_h\cap B$ and $F\cap \interior B=\Sigma_h\cap \interior B$. \end{lemma}

\begin{proof} Suppose that $h^i:\RR^n\rightarrow\RR$ is a harmonic polynomial for each $i\geq 1$ and $h^i\rightarrow h$ in coefficients. Then $h^i\rightarrow h$ uniformly on compact subsets of $\RR^n$. Hence $h$ is harmonic.

Now suppose that, in addition, $h$ is nonconstant and $\Sigma_h\cap \interior B\neq\emptyset$ for some closed ball $B$. Since $\Sigma_h\cap\interior B\neq\emptyset$, we can find a ball $B'\subset B$ whose center lies in $\Sigma_h\cap B$.  By the mean value property for harmonic functions, \begin{equation} \int_{B'} h(y)dy=0.\end{equation} Because $h$ is not identically zero, there exists $x_+,x_-\in B'$ such that $h(x_+)>0$ and $h(x_-)<0$. Hence, since $h^{i}(x_\pm)\rightarrow h(x_\pm)$, we conclude that $h^{i}(x_+)>0$ and $h^{i}(x_-)<0$ for all sufficiently large $i$, as well. By the intermediate value theorem, $h^{i}$ must vanish somewhere in $B'\subset B$ for all large $i$. That is, $\Sigma_{h^i}\cap B\neq\emptyset$ for all sufficiently large $i$.

Finally suppose that $\Sigma_{h^i}\cap B\rightarrow F$ in the Hausdorff distance for some closed set $F$. Recall that we want to show $F\subset\Sigma_h\cap B$ and $F\cap\interior B=\Sigma_h\cap\interior B$. On one hand, for every $y\in F$ there exists $y_i\in B$ such that $h^{i}(y_i)=0$ and $y_i\rightarrow y$. To show $h(y)=0$, consider \begin{equation}h(y)=h(y)-h(y_i)+h(y_i)-h^{i}(y_i)+h^{i}(y_i)
\end{equation} Since $h$ is continuous and $y_i\rightarrow y$, we get $\limsup_{i\rightarrow\infty} |h(y)-h(y_i)|=0$. Because $h^{i}\rightarrow h$ uniformly on $B$, we conclude $\limsup_{i\rightarrow\infty} |h(y_i)-h^{i}(y_i)|\leq \limsup_{i\rightarrow\infty} \|h-h^{i}\|_{L^\infty(B)}=0$. Hence $h(y)=0$ and $F\subset \Sigma_h\cap B$. In particular, $F\cap \interior B\subset \Sigma_h\cap\interior B$.

On the other hand, suppose that $y\in \Sigma_h\cap \interior B$. Choose $m\geq 1$ such that $B(y,1/m)\subset \interior B$. Since $h$ is not identically zero, we can use the mean value property of $h$ as above to show that there exist $y^m_+,y^m_-\in B(y,1/m)$ such that $h(y^m_+)>0$ and $h(y^m_-)<0$. But $h^{i}(y^m_\pm)\rightarrow h(y^m_\pm)$, so there is $i_0$ such that $h^{i}(y^m_+)> 0$ and $h^{i}(y^m_-)<0$ for all $i\geq i_0$. By continuity, we conclude that for each $i\geq i_0$ there exists $y^{m,i}_0\in B(y,1/m)$ such that $h^{i}(y^{m,i}_0)=0$. Thus \begin{equation}\begin{split}
\dist (y,F)&\leq \limsup_{i\rightarrow\infty}\, \dist(y,\Sigma_{h^{i}}\cap B)+\HD[\Sigma_{h^{i}}\cap B, F]\\
&\leq \frac{1}{m}+\limsup_{k\rightarrow\infty}\, \HD[\Sigma_{h^{i}}\cap B, F]= \frac{1}m.\end{split}\end{equation} Letting $m\rightarrow\infty$ yields $\dist(y,F)=0$. But $F$ is closed, so $y\in F$ and $\Sigma_h\cap\interior B\subset F\cap \interior B$. Therefore, $F\cap \interior B=\Sigma_h\cap \interior B$, as desired. \end{proof}

\begin{corollary}\label{hpzeroconverge} Suppose that $(h^i)_{i=1}^\infty$ is a sequence of harmonic polynomials and $h^i\rightarrow h$ in coefficients. If $h$ is nonconstant and $B$ is a closed ball such that $\Sigma_h\cap \interior B\neq\emptyset$ and \begin{equation}\label{notisolated} \Sigma_h\cap B=\overline{\Sigma_h\cap\interior B},\end{equation} then $\Sigma_{h^i}\cap B\rightarrow\Sigma_h\cap B$ in the Hausdorff distance.  \end{corollary}

\begin{proof} Let $(h^i)_{i=1}^\infty$ be any sequence of harmonic polynomials in $\RR^n$ such that $h^i\rightarrow h$ in coefficients to a nonconstant harmonic polynomial $h$. Suppose that $B$ is a closed ball such that $\Sigma_h\cap\interior B\neq\emptyset$ and such that (\ref{notisolated}) holds. By Lemma \ref{hpfconverge}, $\Sigma_{h^i}\cap B\neq\emptyset$ for all sufficiently large $i$. Pick an arbitrary subsequence $(h^{ij})_{j=1}^\infty$ of $(h^i)_{i=1}^\infty$. By Blaschke's selection theorem, we can find a further subsequence $(h^{ijk})_{k=1}^\infty$ of $(h^{ij})_{j=1}^\infty$ and a nonempty closed set $F\subset B$ such that $\lim_{k\rightarrow\infty}\HD[\Sigma_{h^{ijk}}\cap B,F]=0$.  By Lemma \ref{hpfconverge}, $F\subset \Sigma_h\cap B$ and $F\cap\interior B=\Sigma_h\cap\interior B$. Thus, since $B$ satisfies (\ref{notisolated}) and $F$ is closed, \begin{equation}\Sigma_h\cap B=\overline{\Sigma_h\cap \interior B}= \overline{F\cap \interior B}\subset \overline{F}=F\end{equation} This shows $F=\Sigma_h\cap B$.
We have proved every subsequence $(h^{ij})_{j=1}^\infty$ of $(h^i)_{i=1}^\infty$ has a further subsequence $(h^{ijk})_{k=1}^\infty$ such that $\lim_{k\rightarrow\infty}\HD[\Sigma_{h^{ijk}}\cap B,\Sigma_h\cap B]=0$. Therefore, the original sequence $\Sigma_{h^i}\cap B$ also converges to $\Sigma_h\cap B$ in the Hausdorff distance. \end{proof}

\begin{corollary}\label{hptohomog} Suppose that $h:\RR^n\rightarrow\RR$ is a harmonic polynomial of degree $d\geq 1$, and let $x\in\Sigma_h$. If $h(y)=h^{(x)}_d(y-x)+h^{(x)}_{d-1}(y-x)+\dots+h^{(x)}_j(y-x)$ where $h_j^{(x)}\neq 0$, then the unique blow-up of $\Sigma_h$ at $x$ is the zero set $\Sigma_{h^{(x)}_j}$ of $h^{(x)}_j$. That is, \begin{equation}\lim_{r_i\rightarrow 0}\HD\left[\frac{\Sigma_h-x}{r_i}\cap B_s,\Sigma_{h^{(x)}_j}\cap B_s\right]=0\quad\text{ for all }s>0.\end{equation}\end{corollary}

\begin{proof} Suppose that $h(y)=h_d^{(x)}(y-x)+\dots+h_j^{(x)}(y-x)$ is a harmonic polynomial in $\RR^n$ with $j\geq 1$ and $h^{(x)}_j\neq 0$. Given $r_i\downarrow 0$, define $h^i(y)=r_i^{-j}h(x+r_iy)$ for all $y\in\RR^n$. Then $h^i:\RR^n\rightarrow\RR$ is a harmonic polynomial and $r_i^{-1}(\Sigma_h-x)=\Sigma_{h^i}$. Moreover, \begin{equation}\label{elephant1} h^i(y)=r_i^{-j}\left(h^{(x)}_d(r_iy)+\dots+h^{(x)}_j(r_iy)\right)=r^{d-j}_ih^{(x)}_d(y)+\dots+ h^{(x)}_j(y).\end{equation} Since $r_i\rightarrow 0$ as $i\rightarrow\infty$, $h^i\rightarrow h^{(x)}_j$ in coefficients. Because $h_j^{(x)}$ is homogeneous, $\Sigma_{h_j^{(x)}}\cap B_s$ satisfies (\ref{notisolated}). Thus the claim follows immediately from Corollary \ref{hpzeroconverge}.\end{proof}

\begin{remark} The proofs of Lemma \ref{hpfconverge} and Corollary \ref{hpzeroconverge} did not use the full strength of the harmonic property of $h$. Instead we only needed to assume that $h_i$ are polynomials, $h_i\rightarrow h$ in coefficients, and $h$ is a nonconstant polynomial such that for all $x\in\Sigma_h$ and for all $r>0$ there exist $x_+,x_- \in B(x,r)$ such that $h(x_+)>0$ and $h(x_-)<0$.
\end{remark}

\section{Local Flatness of Zero Sets of Polynomials}
\label{SectFlatnessZeroSets}

The local flatness of the zero set $\Sigma_p$ of a polynomial at a root $x$ is controlled from above by the relative size $\zeta_1(p,x,r)$ of the linear term $p^{(x)}_1$.

\begin{lemma}\label{zetaimpliesflat} If $p:\RR^n\rightarrow\RR$ is a polynomial of degree $d\geq 1$ such that $p(x)=0$, then $\theta_{\Sigma_p}(x,r)\leq C_d\zeta_1(p,x,r)$ for all $r>0$. Explicitly, $C_d=\sqrt{2}(d-1)$.\end{lemma}

\begin{proof} The claim is trivial for polynomials of degree 1. Let $p:\RR^n\rightarrow\RR$ be any polynomial of degree $d\geq 2$, let $x\in\Sigma_p$ and let $r>0$. For the proof we may assume that \begin{equation}\label{zebra1}\sqrt{2}(d-1)\zeta_1(p,x,r)\leq 1,\end{equation} since the bound $\theta_{\Sigma_p}(x,r)\leq 1$ is always true. Because $\zeta_1(p,x,r)<\infty$, we know that $p_1^{(x)}\neq 0$, by Lemma \ref{maglemma}. Hence $L=\{p_1^{(x)}=0\}\in G(n,n-1)$ is an $(n-1)$-dimensional plane through the origin. Let $e$ be the unique unit normal vector to $L$ at $0$ such that $p_1^{(x)}(e)>0$, and set $\delta:=(d-1)\zeta_1(p,x,r)$. If $y\in L$ and $t>\delta$ satisfy $|y+tre|\leq r$, then \begin{equation}\begin{split} p(x+y+t re)&=p_d^{(x)}(y+t r e)+\dots+ p_1^{(x)}(y+t r e)\\
&\geq p_1^{(x)}(y+t r e)- \|p_d^{(x)}\|_{L^\infty(B_r)}-\dots-\|p_2^{(x)}\|_{L^\infty(B_r)}\\
&\geq p_1^{(x)}(y+t r e)-(d-1)\zeta_1(p,x,r)\|p_1^{(x)}\|_{L^\infty(B_r)}\\
&= t\|p_1^{(x)}\|_{L^\infty(B_r)}-\delta\|p_1^{(x)}\|_{L^\infty(B_r)}>0.\end{split}\end{equation} Similarly, $p(x+y+tre)<0$ when $|y+tre|\leq r$ and  $t< -\delta$. Hence every root of $p$ in $B(x,r)$ lies in the strip $\{x+y+tre:y\in L\text{ and } |t|\leq \delta\}$. Thus $\dist(z,x+L)\leq \delta r$ for all $z\in \Sigma_p\cap B(x,r)$.

On the other hand, suppose that $y\in L\cap B_r$. Then we can connect $y$ by two line segments $\ell_\pm=[y,z_\pm]$ to $(L\pm\delta r e)\cap B_r$ of minimal length (see Figure 4.1). Since $p(x+z_+)\geq 0$ and $p(x+z_-)\leq 0$, by continuity $p(x+z_0)$ must vanish at some point $z_0\in\ell_+\cup\ell_-$. Hence $\dist(x+y,\Sigma_p\cap B(x,r))$ is bounded above by the length of $\ell_\pm$. This is a geometric constant, which is at worst $\sqrt{2}r(1-\sqrt{1-\delta^2})^{1/2}$. (To compute this, notice the length of $\ell_\pm$ is at worst the distance of $(r,\overline{0},0)$ to $(tr,\overline{0},\delta r)$ where $t^2+\delta^2=1$.) Since $0\leq \delta\leq 1$, it follows that  $1-\delta^2\leq \sqrt{1-\delta^2}$ and the length of $\ell_\pm$ is bounded above by $\sqrt{2}r(1-(1-\delta^2))^{1/2}=\sqrt{2}\delta r$.
Thus, $\dist(x+y,\Sigma_p\cap B(x,r))\leq \sqrt{2}\delta r$ for every $x+y\in x+L$. Therefore, we conclude that $\theta_{\Sigma_p}(x,r)\leq \sqrt{2}\delta=\sqrt{2}(d-1)\zeta_1(p,x,r)$, as desired.\end{proof}

\begin{figure}\label{plemmafig}\begin{center}\includegraphics[width=.4\textwidth]{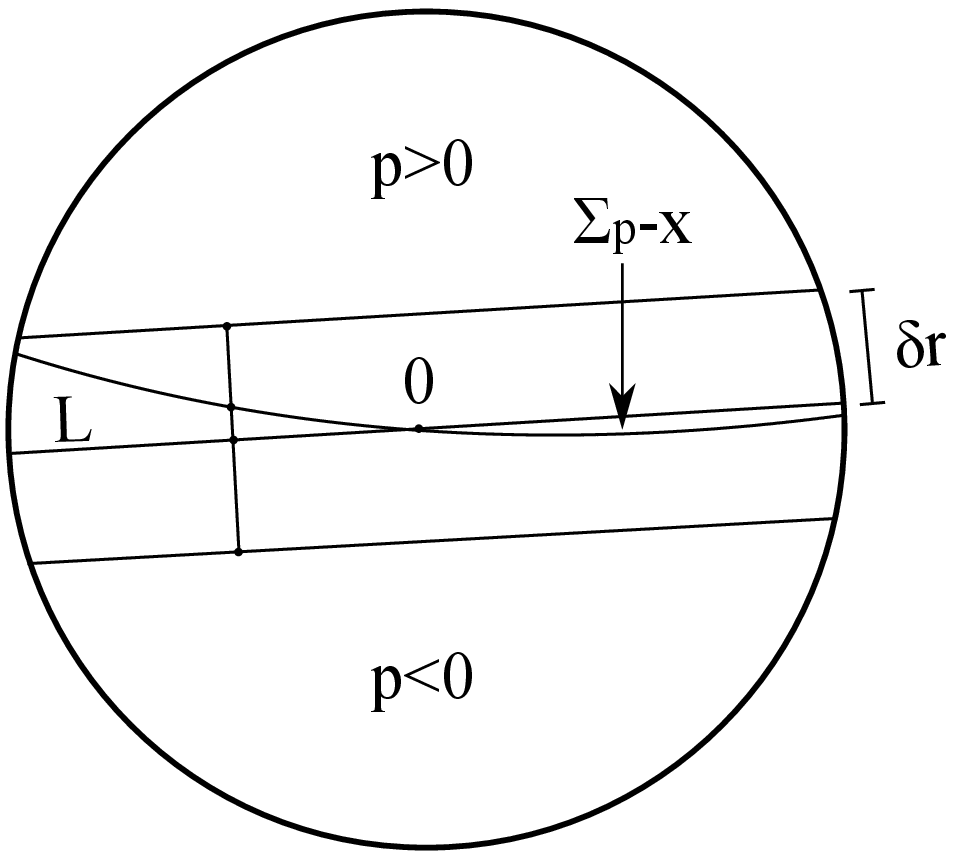}\end{center}\caption{Proof of Lemma \ref{zetaimpliesflat}}\end{figure}

The converse of Lemma \ref{zetaimpliesflat} does not hold in general. Indeed if $p(x,y)=x^4+y^4-y^2$, then $\zeta_1(p,0,r)=\infty$ for all $r>0$ even though $\lim_{r\rightarrow 0}\theta_{\Sigma_p}(0,r)=0$.
Nevertheless we can establish a converse to Lemma \ref{zetaimpliesflat} for harmonic polynomials! As an intermediate step, we first consider homogeneous harmonic polynomials. The following auxiliary estimates for spherical harmonics play a key role.

\begin{definition} A \emph{spherical harmonic} $h:S^{n-1}\rightarrow\RR$ of \emph{degree} $k$ is the restriction of a homogeneous harmonic polynomial $h:\RR^n\rightarrow\RR$ of degree $k$ to the unit sphere.\end{definition}

\begin{remark}\label{sphextrm} If $h:S^{n-1}\rightarrow\RR$ is a spherical harmonic, then there may exist distinct polynomials $p:\RR^n\rightarrow\RR$ and $q:\RR^n\rightarrow\RR$ such that $p|_{S^{n-1}}=q|_{S^{n-1}}=h$. For instance, the polynomials $p(x)=1$ and $q(x)=|x|^2=x_1^2+\dots+x_n^2$ agree on $S^{n-1}$. Nevertheless, there always exists a unique (homogeneous) harmonic polynomial $\tilde h:\RR^n\rightarrow\RR$ such that $\tilde h|_{S^{n-1}}=h$.\end{remark}

Using well-known local estimates for the derivatives of harmonic functions, one can prove that uniformly bounded spherical harmonics of degree $k$ have uniform Lipschitz constant.

\begin{proposition}[\cite{Badger1} Proposition 3.2]\label{spharmlip} For every $n\geq 2$ and $k\geq 1$ there exists a constant $A_{n,k}>1$ such that for every spherical harmonic $h:S^{n-1}\rightarrow\RR$ of degree $k$, \begin{equation}\label{beachball1}|h(\theta_1)-h(\theta_2)|\leq A_{n,k}\|h\|_{L^\infty(S^{n-1})}|\theta_1-\theta_2|\quad\text{for all }\theta_1,\theta_2\in S^{n-1}.\end{equation}\end{proposition}

\begin{corollary}\label{zerodist} For every spherical harmonic $h:S^{n-1}\rightarrow\RR$ of degree $k\geq 1$, \begin{equation}\label{beachball2} |h(\theta)|\leq A_{n,k}\|h\|_{L^\infty(S^{n-1})}\dist(\theta,\Sigma_h)\quad\text{for all }\theta\in S^{n-1}\end{equation}\end{corollary}

\begin{proof} Apply Proposition \ref{spharmlip} with $\theta_1=\theta$ and $\theta_2\in\Sigma_h\cap S^{n-1}$. Then minimizing (\ref{beachball1}) over $\theta_2\in\Sigma_h\cap S^{n-1}$ yields (\ref{beachball2}).\end{proof}

\begin{corollary}\label{halfdist} Let $h:S^{n-1}\rightarrow\RR$ be a spherical harmonic of degree $k\geq 1$. If $\theta_0\in S^{n-1}$ satisfies $|h(\theta_0)|=\|h\|_{L^\infty(S^{n-1})}$, then \begin{equation}\label{beachball3}|h(\theta)|> \frac12\|h\|_{L^\infty(S^{n-1})}\quad\text{ for every }\theta\in S^{n-1}\cap B(\theta_0, 1/2A_{n,k}).\end{equation}\end{corollary}

\begin{proof} By the reverse triangle inequality and Proposition \ref{spharmlip}, \begin{equation}|h(\theta)| \geq |h(\theta_0)|-|h(\theta_0)-h(\theta)|\geq |h(\theta_0)| - A_{n,k}\|h\|_{L^\infty(S^{n-1})} |\theta_0-\theta|.\end{equation} Thus $|h(\theta)|\geq \frac{1}{2}\|h\|_{L^\infty(S^{n-1})}$ whenever $|h(\theta_0)|=\|h\|_{L^\infty(S^{n-1})}$ and $|\theta_0-\theta|\leq 1/2A_{n,k}$.\end{proof}

The following lemma says that at the origin the zero set of a homogeneous harmonic polynomial of degree $k\geq 2$ is far away from flat.

\begin{lemma}\label{hlemma} For every $n\geq 2$ and $k\geq 2$ there is a constant $\delta'_{n,k}$ with the following property. For every homogeneous harmonic polynomial $h:\RR^n\rightarrow\RR$ of degree $k$ and for every scale $r>0$, $\theta_{\Sigma_h}(0,r)\geq \delta'_{n,k}$.\end{lemma}

\begin{figure}\begin{center}\includegraphics[width=.8\textwidth]{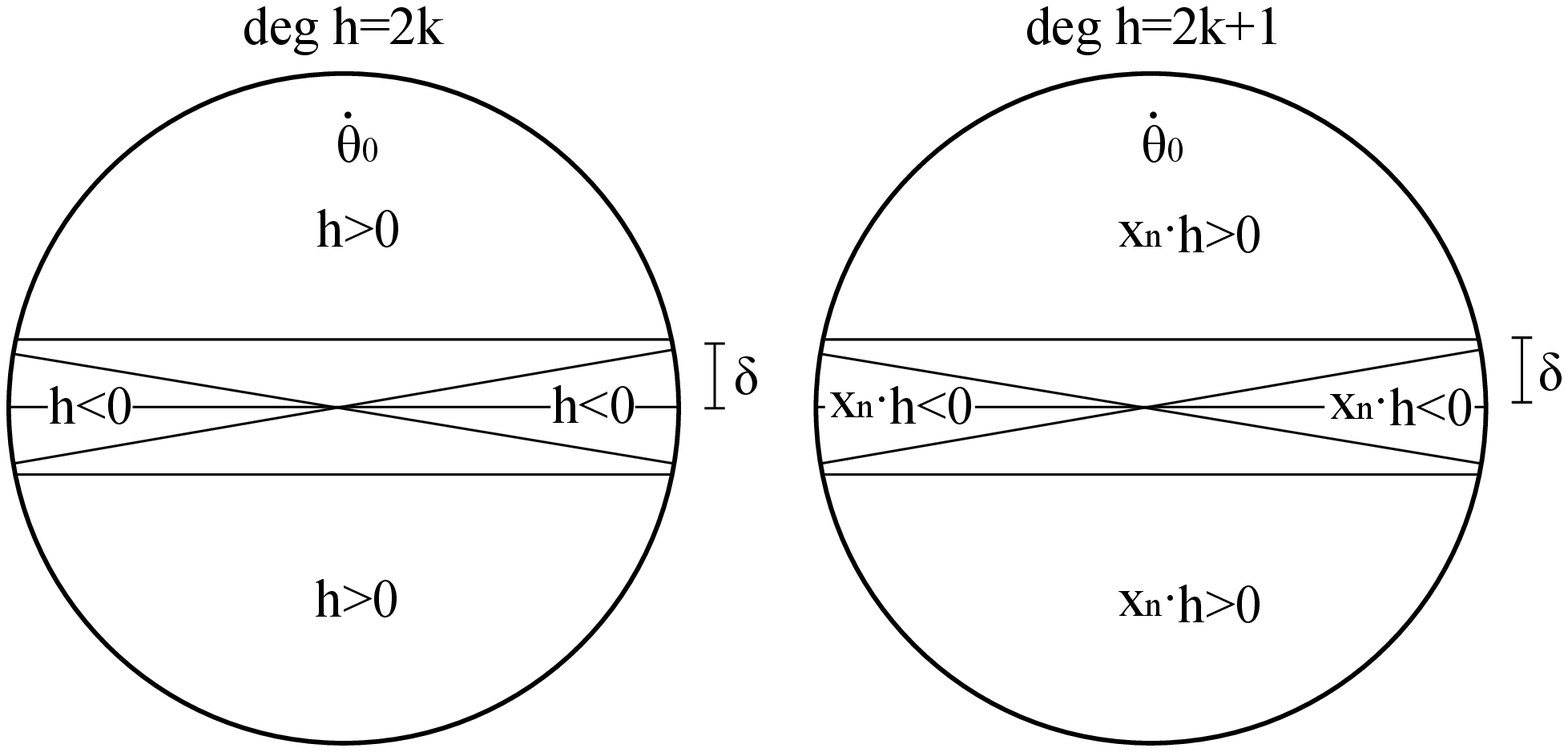}\end{center}\caption{Proof of Lemma \ref{hlemma}}\end{figure}

\begin{proof} Suppose that $h:\RR^n\rightarrow\RR$ is a homogeneous harmonic polynomial of degree $k$. Since $h$ is homogeneous, $\delta=\theta_{\Sigma_h}(0,1)=\theta_{\Sigma_h}(0,r)$ for all $r>0$. By applying a rotation, we may assume without loss of generality that $\HD[\Sigma_h\cap B(0,1),\{x_n=0\}\cap B(0,1)]\leq \delta$. Also, by replacing $h$ with $-h$ if necessary, we may assume that there exists $\theta_0\in S^{n-1}$ such that $h(\theta_0)=\|h\|_{L^\infty(S^{n-1})}$ (i.e.\ the sup norm is obtained at a positive value of $h$). Finally, by performing a change of coordinates $x\mapsto -x$ if necessary, we may assume $\theta_0\in\RR^{n-1}\times\RR^+$ (i.e.\ the last coordinate of $\theta_0$ is positive). We now break the argument into two cases, depending on the parity of $k$.

Suppose that $k\geq 2$ is even. By the mean value property for harmonic functions, \begin{equation}\label{giraffe1}\frac{1}{\sigma_{n-1}}\int_{S^{n-1}}h(\theta)d\surf(\theta)=h(0)=0.\end{equation} We will show that $\delta$ being small violates (\ref{giraffe1}). By Corollary \ref{zerodist}, $\dist(\theta_0,\Sigma_h)\geq A_{n,k}^{-1}$. Hence $h(\theta)>0$ for all $\theta\in S^{n-1}\cap \{x_n>\delta\}$ provided that $\delta \ll A_{n,k}^{-1}$ (the last coordinate of $\theta_0$ is positive). Assume this is true. Since $h$ is even, $h(\theta)>0$ for all $\theta\in S^{n-1}\cap \{x_n<-\delta\}$, as well. Thus negative values of $h$ (obtained at points of the sphere) can only be obtained inside the strip $S_\delta=S^{n-1}\cap \{|x_n|\leq \delta\}$. Moreover, $|h(\theta)|\leq 2\delta A_{n,k}\|h\|_{L^\infty(S^{n-1})}$ for $\theta\in S_\delta$, by Corollary \ref{zerodist}. But $h(\theta)\geq (1/2)\|h\|_{L^\infty(S^{n-1})}$ for all $\theta\in \Delta_0=S^{n-1}\cap B(\theta_0,1/2A_{n,k})$ by Corollary \ref{halfdist}. It follows that \begin{equation}\begin{split} \int_{S^{n-1}} h(\theta) d\surf(\theta) &= \int_{S^{n-1}} h^+(\theta) d\surf(\theta) - \int_{S^{n-1}} h^-(\theta) d\surf(\theta) \\ &\geq \int_{\Delta_0} h(\theta)d\surf(\theta)-\int_{S_\delta}|h(\theta)|d\surf(\theta)\\
&\geq \|h\|_{L^\infty(S^{n-1})}\left(\frac12\surf(\Delta_0)-2\delta A_{n,k}\surf(S_\delta)\right)>0\end{split}\end{equation} if $\delta$ is too small, for example if $\delta< \surf(\Delta_0)/8A_{n,k}\sigma_{n-1}=\delta_{n,k}'$, which violates (\ref{giraffe1}). Therefore, $\delta\geq \delta'_{n,k}$ when $k$ is even.

Suppose that $k\geq 3$ is odd. Because the spherical harmonics of different degrees are orthogonal in $L^2(S^{n-1})$ (e.g.\ see \cite{HFT} Proposition 5.9), \begin{equation}\label{giraffe2} \int_{S^{n-1}} \theta_n h(\theta)d\surf(\theta)=0.\end{equation} This time we will show that (\ref{giraffe2}) is violated if $\delta$ is small. Since $\dist(\theta_0,\Sigma_h)\geq A_{n,k}^{-1}$, $h(\theta)>0$ and $\theta_n h(\theta)>0$ for all $\theta\in S^{n-1}\cap \{x_n>\delta\}$ if $\delta\ll A_{n,k}^{-1}$. Assume this is true. Since $h$ is odd, $\theta_n h(\theta)$ is even and $\theta_n h(\theta)>0$ for all $\theta\in S^{n-1}\cap\{x_n<-\delta\}$ too. Hence $\theta_n h(\theta)$ can only assume negative values in the strip $S_\delta=S^{n-1}\cap \{|x_n|\leq\delta\}$. Moreover, $|\theta_n h(\theta)|\leq 2\delta^2 A_{n,k}\|h\|_{L^\infty(S^{n-1})}$ for every $\theta\in S_\delta$, by Corollary \ref{zerodist}. On the other hand, $\theta_n h(\theta)>\delta(1/2)\|h\|_{L^\infty(S^{n-1})}$ for all $\theta\in \Delta_0=S^{n-1}\cap B(\theta_0,1/2A_{n,k})$, by Corollary \ref{halfdist}. Thus \begin{equation}\begin{split} \int_{S^{n-1}} \theta_nh(\theta) d\surf(\theta) &= \int_{S^{n-1}} (\theta_nh(\theta))^+ d\surf(\theta) - \int_{S^{n-1}} (\theta_nh(\theta))^- d\surf(\theta) \\ &\geq \int_{\Delta_0} \theta_n h(\theta)d\surf(\theta)-\int_{S_\delta}|\theta_nh(\theta)|d\surf(\theta)\\
&\geq \delta\|h\|_{L^\infty(S^{n-1})}\left(\frac12\surf(\Delta_0)-2\delta A_{n,k}\surf(S_\delta)\right)>0\end{split}\end{equation} if $\delta$ is too small, for example if $\delta< \surf(\Delta_0)/8A_{n,k}\sigma_{n-1}=\delta_{n,k}'$, which violates (\ref{giraffe2}). Therefore, $\delta\geq \delta'_{n,k}$ when $k$ is odd.\end{proof}

Harmonic polynomials enjoy a partial converse to Lemma \ref{zetaimpliesflat}.

\begin{proposition}\label{converselemma} For all $n\geq 2$ and $d\geq 1$ there exist $\delta_{n,d}>0$ with the following property. If $h:\RR^n\rightarrow\RR$ is a harmonic polynomial of degree $d$ and $h(x)=0$, then $\zeta_1(h,x,r)<\delta_{n,d}^{-1}$ whenever $\theta_{\Sigma_h}(x,r)<\delta_{n,d}$.\end{proposition}

\begin{proof} If $d=1$, then $h$ is linear and $\zeta_1(h,x,r)=0=\theta_{\Sigma_h}(x,r)$ for all $x\in\Sigma_h$.  Thus the case $d=1$ is trivial.

Let $n\geq 2$ and $d\geq 2$ be fixed. Suppose for contradiction that for every $N\geq 1$ there exists a harmonic polynomial $h^{N}:\RR^n\rightarrow\RR$ of degree $d$, $x_N\in\RR^n$ and $r_N>0$ such that $h^{N}(x_N)=0$, $\zeta_1(h^N,x_N,r_N)>N$ and $\theta_{\Sigma_{h^N}}(x_N,r_N)<1/N$. Replacing each polynomial $h^N$ with $\tilde h^N(y)=c_Nh(r_N(y+x_N))$, we may assume without loss of generality that $x_N=0$ and $r_N=1$ for all $N\geq 1$, and $\max_{|\alpha|\leq d}|D^\alpha h^N(0)|=1$. Thus, there exists a sequence $h^N$ of harmonic polynomials in $\RR^n$ of degree $d$ with uniformly bounded coefficients such that $h^N(0)=0$, $\zeta_1(h^N,0,1)\geq N$ and $\theta_{\Sigma_{h^N}}(0,1)\leq 1/N$ for all $N\geq 1$. Passing to a subsequence, we may assume that $h^N\rightarrow h$ in coefficients to some harmonic polynomial $h:\RR^n\rightarrow\RR$. Note  $h$ is nonconstant because we assumed that for each polynomial $h^N$ there is some multi-index $\alpha$ such that $|D^\alpha h^N(0)|=1$. Also note $\zeta_1(h,0,1)=\infty$, by Lemma \ref{zetacts}. Taking a further subsequence we may also assume that there exists a closed set $F$ such that $\Sigma_{h^N}\cap B_1\rightarrow F$ in the Hausdorff distance. By Lemma \ref{hpfconverge}, $F\cap \interior B_1=\Sigma_h\cap\interior B_1$. Hence $\theta_{\Sigma_h}(0,r)=0$ for all $r<1$. To complete the proof we blow up $\Sigma_h$ at the origin and apply Lemma \ref{hlemma}. Expand $h$ as $h=h^{(0)}_k+\dots + h^{(0)}_j$ where $k=\deg h$ and $h^{(0)}_j\neq 0$. Note $2\leq j\leq k$, since $\zeta_1(h,0,1)=\infty$. Choose any sequence $r_i\downarrow 0$ and define $h^i(y)=h(r_i y)$ for all $y\in\RR^n$. By Corollary \ref{hptohomog}, $\Sigma_{h^{i}}\cap B_1\rightarrow \Sigma_{h^{(0)}_j}\cap B_1$ in the Hausdorff distance. Therefore, since $\theta_{\Sigma_{h^i}}(0,1)=\theta_{\Sigma_{h}}(0,r_i)=0$ for all $i$ such that $r_i<1$, we conclude that $\theta_{\Sigma_{h^{(0)}_j}}(0,1)=0$. Since $h^{(0)}_j$ is a homogeneous polynomial of degree $j\geq 2$, this contradicts Lemma \ref{hlemma}. Our supposition was false. Hence there exists $N_{n,d}\geq 1$ such that for every harmonic polynomial $h$ in $\RR^n$ of degree $d$, every $x\in\Sigma_{h}$ and every $r>0$, $\zeta_1(h,x,r)<N_{n,d}$ whenever $\theta_{\Sigma_h}(x,r)<1/N_{n,d}=\delta_{n,d}$.\end{proof}

\begin{corollary}\label{epcor1} Let $h:\RR^n\rightarrow\RR$ be a harmonic polynomial of degree $d\geq 1$. If $h(x)=0$ and $\theta_{\Sigma_h}(x,r)<\delta_{n,d}$, then $Dh(x)\neq 0$ and $\theta_{\Sigma_h}(x,sr)< C_{n,d}s$ for all $s\in(0,1)$.\end{corollary}

\begin{proof} Once again there is nothing to prove if $d=1$. Thus assume that $h:\RR^n\rightarrow\RR$ is a harmonic polynomial of degree $d\geq 2$ such that $\theta_{\Sigma_h}(x,r)<\delta_{n,d}$ for some $x\in\Sigma_h$ and $r>0$. By Proposition \ref{converselemma}, we get $\zeta_1(h,x,r)<\delta_{n,d}^{-1}$. Hence, because $\zeta_1(h,x,r)<\infty$, Lemma \ref{maglemma} guarantees that $h_1^{(x)}\neq 0$, or equivalently, $Dh(x)\neq 0$. Using Lemma \ref{lineardecay}, we conclude that $\zeta_1(h,x,sr)\leq s\zeta_1(h,x,r)<s\delta_{n,d}^{-1}$ for all $s\in(0,1)$. Thus, by Lemma \ref{zetaimpliesflat}, $\theta_{\Sigma_h}(x,sr)<\sqrt{2}(d-1)\delta_{n,d}^{-1}s=C_{n,d}s$ for all $s\in(0,1)$.\end{proof}

\begin{corollary}\label{epcor2} Let $h:\RR^n\rightarrow\RR$ be a harmonic polynomial of degree $d\geq 1$. If $x\in\Sigma_h$ and $Dh(x)=0$, then $\theta_{\Sigma_h}(x,r)\geq \delta_{n,d}$ for all $r>0$.\end{corollary}

\begin{proof} This is the contrapositive of Corollary \ref{epcor1}.\end{proof}

We can now record the proof of Theorem \ref{polythm}. Recall: \emph{For all $n\geq 2$ and $d\geq 1$ there exists a constant $\delta_{n,d}>0$ such that for any harmonic polynomial $h:\RR^n\rightarrow\RR$ of degree $d$ and for any $x\in\Sigma_h$,
\begin{eqnarray*} Dh(x)=0\quad&\Leftrightarrow\quad&\theta_{\Sigma_h}(x,r)\geq \delta_{n,d}\quad\hbox{for all }r>0,\\
 Dh(x)\neq 0\quad&\Leftrightarrow\quad&\theta_{\Sigma_h}(x,r)< \delta_{n,d}\quad\hbox{for some }r>0.\end{eqnarray*} Moreover, there exists a constant $C_{n,d}>1$ such that if $\theta_{\Sigma_h}(x,r)<\delta_{n,d}$ for some $r>0$, then $\theta_{\Sigma_h}(x,sr)< C_{n,d} s$ for all $s\in(0,1)$.}

\begin{proof}[of Theorem \ref{polythm}] Suppose that $h:\RR^n\rightarrow\RR$, $n\geq 2$, is a harmonic polynomial of degree $d\geq 1$ and fix $x\in\Sigma_h$. If $Dh(x)\neq 0$, then $\zeta_1(h,x,1)<\infty$ by Lemma \ref{maglemma}. Applying Lemma \ref{zetaimpliesflat} and Lemma \ref{lineardecay}, it follows that \begin{equation} \theta_{\Sigma_h}(x,r)\leq \sqrt{2}(d-1)r\zeta_1(h,x,1)\quad\text{for all }r\in(0,1).\end{equation} Since $\zeta_1(h,x,1)<\infty$, we see that $\theta_{\Sigma_h}(x,r)<\delta_{n,d}$ for some $r\in(0,1)$ sufficiently small. Conversely, if $Dh(x)=0$, then $\theta_{\Sigma_h}(x,r)\geq\delta_{n,d}$ for all $r>0$, by Corollary \ref{epcor2}. Finally, if $\theta_{\Sigma_h}(x,r)<\delta_{n,d}$, then $\theta_{\Sigma_h}(x,sr)<C_{n,d}s$ for all $s\in(0,1)$, by Corollary \ref{epcor1}.
\end{proof}

\begin{remark} It is natural to ask if a stronger statement than Proposition \ref{converselemma} holds. Namely, is it true that given $\varepsilon>0$ there exists $\delta>0$ such that $\theta_{\Sigma_h}(x,r)<\delta\Rightarrow\zeta_1(h,x,r)<\varepsilon$? Unfortunately the answer is no, as the following example illustrates. Consider the harmonic polynomial $h(x,y)=xy$ with root $(2,0)\in\Sigma_h$ and radius $r=1$. On one hand, $$\Sigma_h \cap B((2,0),1)=[1,3]\times\{0\}$$ is a line segment: $\theta_{\Sigma_h}((2,0),1)=0$. On the other hand, $h((x,y)+(2,0))=xy+2y$. Hence  $h^{(2,0)}_1=2y$, $h^{(2,0)}_2=xy$ and $\zeta_1(h,(2,0),1)=\|xy\|_{L^\infty(B_1)}/\|2y\|_{L^\infty(B_1)}=1/2>0$. This shows $\zeta_1(h,(2,0),1)>0$ even though $\theta_{\Sigma_h}((2,0),1)<\delta$ for all $\delta>0$.\end{remark}

\begin{remark}\label{sharmseparate} In Lemma \ref{hlemma} we showed that the zero set $\Sigma_{g}$ of a homogeneous harmonic polynomial $g$ of degree $1$ (that is, a hyperplane through the origin) and the zero set $\Sigma_h$ of a homogeneous harmonic polynomial of degree $d\geq 2$ are far apart in Hausdorff distance: \begin{equation}\HD[\Sigma_g\cap B_1,\Sigma_h\cap B_1]\geq \delta(1,n)>0.\end{equation} However, the following questions remain open. Is it true that, whenever $g$ and $h$ are homogeneous harmonic polynomials of different degrees, say $c=\deg g<\deg h=d$, then their zero sets are far apart in the sense that $\HD[\Sigma_g\cap B_1,\Sigma_h\cap B_1]\geq \delta(c,d)>0$? Furthermore, provided the answer is yes, is this separation independent of the degrees in the sense that $\inf\{\delta(m,n):1\leq c<d\}>0$?
\end{remark}

\section{Sets Approximated by Zero Sets of Harmonic Polynomials}
\label{SectApproximation}

In this section, we study flat points in sets which admit uniform local approximations by zero sets of harmonic polynomials. To make this idea of local approximation precise, the following notation will prove useful.

\begin{definition}[(Local Set Approximation)] \label{lapprox} \hspace{1cm} \begin{enumerate}
 \item[(i)] A \emph{local approximation class} $\mathcal{S}$ is a nonempty collection of closed subsets of $\RR^n$ such that $0\in S$ for each $S\in\mathcal{S}$. For every closed set $A\subset\RR^n$, $x\in A$ and $r>0$, define the \emph{local closeness} $\Theta^{\mathcal{S}}_A(x,r)$ of $A$ to $\mathcal{S}$ near $x$ at scale $r$ by \begin{equation*} \Theta^{\mathcal{S}}_A(x,r)=\frac{1}{r}\inf_{S\in \mathcal{S}}\HD\left[A\cap B(x,r), (x+S)\cap B(x,r)\right]\end{equation*}
\item[(ii)] A closed set $A\subset\RR^n$ is said to be \emph{locally $\delta$-close to $\mathcal{S}$} if for every compact set $K\subset\RR^n$ there is $r_0>0$ such that $\Theta^{\mathcal{S}}_A(x,r)\leq \delta$ for all $x\in A\cap K$ and $0<r\leq r_0$.
\item[(iii)] A closed set $A\subset\RR^n$ is \emph{locally well-approximated by $\mathcal{S}$} if $A$ is locally $\delta$-close to $\mathcal{S}$ for all $\delta>0$.
\end{enumerate}\end{definition}

\begin{example} \label{ReifDefn} Let $\mathcal{F}=G(n,n-1)$ be the Grassmannian of $(n-1)$-dimensional subspaces of $\RR^n$. Note $\mathcal{F}$ is a local approximation class and $\Theta^{\mathcal{F}}_A(x,r)=\theta_{A}(x,r)$ is the \emph{local flatness} of $A$ near $x$ at scale $r$ defined in the introduction. Sets which admit uniform approximations by hyperplanes at all locations and scales first appeared in Reifenberg's solution of the Plateau problem in arbitrary codimension \cite{Reifenberg}; they are now called Reifenberg flat sets. In the terminology of Definition \ref{lapprox}, a closed set $A\subset\RR^n$ is \emph{locally $\delta$-Reifenberg flat} if $A$ is locally $\delta$-close to $\mathcal{F}$; and $A$ is \emph{locally Reifenberg flat with vanishing constant} provided $A$ is locally well-approximated by $\mathcal{F}$.
\end{example}

Let's pause to collect three facts about Reifenberg flat sets and local flatness.

Fact one. Uniform local flatness guarantees good topology of a set:

\begin{theorem}[Reifenberg's topological disk theorem] \label{diskthm} For each $n\geq 2$, there exists $\delta_n>0$ with the following property. If $0<\delta\leq \delta_n$ and $A\subset\RR^n$ is locally $\delta$-Reifenberg flat, then $A$ is locally homeomorphic to an $(n-1)$-dimensional disk.\end{theorem}

In fact, Reifenberg \cite{Reifenberg} proved (but did not state) that a set $A$ in Theorem \ref{diskthm} admits local bi-H\"older parameterizations of a disk. For a recent exposition of the topological disk theorem and its extension to ``sets with holes", see David and Toro \cite{DT}.

Fact two. Uniform local flatness controls the Hausdorff dimension of a set from above, in a quantitative way. (For a complementary statement about the Hausdorff dimension of ``uniformly non-flat'' sets, see Bishop and Jones \cite{BJ} and David \cite{D}.)

\begin{theorem}[Mattila and Vuorinen \cite{MV}] \label{mvdim} If $A\subset\RR^n$ is a locally $\delta$-Reifenberg flat set, then $\dim_HA\leq n-1+C\delta^2$ for some $C=C(n)>0$.\end{theorem}

\begin{corollary}\label{vanishdim} If $A\subset\RR^n$ is locally Reifenberg flat with vanishing constant, then $\dim_HA=n-1$.\end{corollary}

\begin{proof} Let $A\subset\RR^n$ be locally Reifenberg flat with vanishing constant. On one hand, $\dim_HA\leq n-1$ by Theorem \ref{mvdim}. On the other hand, the lower bound $\dim_HA\geq n-1$ follows from Theorem \ref{diskthm}.\end{proof}

Fact three. Given an estimate on the local flatness of a set at one scale, we automatically get (a worse) estimate on the local flatness at a smaller scale for nearby locations.

\begin{lemma}\label{thetaonehalf} Let $A\subset\RR^n$ be a nonempty closed set and let $x,y\in A$. If $B(y,sr)\subset B(x,r)$, then $\theta_A(y,sr) \leq 4 \theta_A(x,r)/s$.\end{lemma}

\begin{proof} Applying a harmless translation, dilation and rotation, we may assume without loss of generality that $x=0$, $r=1$ and \begin{equation} \delta=\theta_A(0,1)=\HD[A\cap B_1,L_0\cap B_1]\end{equation} where $L_0=\{x\in\RR^n:x_n=0\}$. Fix $y\in A$ and $s>0$ such that $B(y,s)\subset B_1$. To estimate $\theta_A(y,s)$ from above we will bound the Hausdorff distance between the set $A\cap B(y,s)$ and the hyperplane $L_y\cap B(y,s)$ inside $B(y,s)$ where the $L_y=\{x\in\RR^n: x_n=y_n\}$.

Suppose that $z\in A\cap B(y,s)$. Since $z\in A\cap B_1$, $\dist(z,L_0\cap B_1)=|z-\pi(z)|\leq \delta$ where $\pi:\RR^n\rightarrow L_0$ denotes the orthogonal projection onto $L_0$. Hence \begin{equation}\dist(z,L_y\cap B(y,s))\leq \delta+\dist(\pi(z), L_y\cap B(y,s)).\end{equation} To continue, note that since $\pi(z)\in \pi(B(y,s))$, $\dist(\pi(z), L_y\cap B(y,s))=|y_n|\leq \delta$. Thus $\dist(z, L_y\cap B(y,s))\leq 2\delta$ for all $z\in A\cap B(y,s)$.

Next suppose that $w\in L_y\cap B(y,s)$. Since $\pi(w)\in L_0\cap B_1$, $\dist (\pi(w), A\cap B_1)\leq \delta$. Hence $\dist(w, A\cap B_1)\leq |w-\pi(w)|+\dist(\pi(w),A\cap B_1)\leq |y_n|+\delta\leq 2\delta$ for every $w\in L_y\cap B(y,s)$. But we really want to estimate $\dist(w, A\cap B(y,s))$. To that end choose $w'\in L_y\cap B(y,s-2\delta)$ such that $|w-w'|\leq 2\delta$. From above we know $\dist(w', A\cap B_1)\leq 2\delta$, say $\dist(w', A\cap B_1)=|w'-x'|\leq 2\delta$ for some $x'\in A\cap B(0,1)$. Because $w'\in B(y,s-2\delta)$, it follows that $|x'-y|\leq |x'-w'|+|w'-y|\leq 2\delta+s-2\delta\leq s$ and $x'\in A\cap B(y,s)$. Thus $\dist(w', A\cap B(y,s))\leq 2\delta$. We conclude \begin{equation}\dist(w, A\cap B(y,s)) \leq |w-w'|+\dist(w',A\cap B(y,s)\leq 4\delta.\end{equation} Therefore, \begin{equation}\label{beaver1} \theta_A(y,s) \leq \frac{1}{s} \HD[A\cap B(y,s), L_y\cap B(y,s)] \leq \frac{4\delta}{s}\end{equation} as desired.  In fact, we have only established (\ref{beaver1}) provided that $L_y\cap B(y,s-2\delta)\neq\emptyset$, i.e.\ when $s>2\delta$. On the other hand, if $s\leq 2\delta$, then $\theta_A(y,s)\leq 1<2\leq 4\delta/s$, as well.\end{proof}

In order to discuss local approximations of a set by zero sets of harmonic polynomials, we introduce a local approximation class $\mathcal{H}_d$.

\begin{definition}\label{HdApprox} For each $d\geq 1$ assign $\mathcal{H}_d$ to be the collection of zero sets $V=h^{-1}(0)$ of nonconstant harmonic polynomials $h:\RR^n\rightarrow\RR$ of degree at most $d$ such that $h(0)=0$. Note $\mathcal{H}_d$ is a local approximation class. If $A\subset\RR^n$ is a closed set, $x\in A$ and $r>0$, then \begin{equation}\label{theta2approx}\Theta^{\mathcal{H}_d}_A(x,r)=\frac{1}{r}\inf_{V} \HD[A\cap B(x,r), (x+V)\cap B(x,r)]\end{equation} where $V$ ranges over the zero sets of harmonic polynomials $h:\RR^n\rightarrow\RR$ such that $h(0)=0$ and $1\leq \deg h\leq d$.\end{definition}

\begin{remark} When $d=1$, $\mathcal{H}_1=\mathcal{F}=G(n,n-1)$. Thus $\Theta^{\mathcal{H}_1}_A(x,r)=\Theta^{\mathcal{F}}_A(x,r)=\theta_A(x,r)$ is the local flatness of $A$ near $x\in A$ at scale $r>0$.\end{remark}

The following lemma roughly states that if a set is uniformly close to the zero set of a harmonic polynomial on all small scales, then local flatness at one scale automatically controls local flatness on smaller scales. This is an application of Theorem \ref{polythm}.

\begin{lemma}\label{emu} For all $n\geq 2$, $d\geq 1$ and $\delta>0$, there exist $\varepsilon=\varepsilon(n,d,\delta)>0$ and $\eta=\eta(n,d,\delta)>0$ with the following property. Let $A\subset\RR^n$, $x\in A$, $r>0$ and assume that \begin{equation}\label{emu1} \sup_{0<r'\leq r}\Theta^{\mathcal{H}_d}_A(x,r')<\varepsilon.\end{equation} If $\theta_A(x,r)<\eta$, then $\sup_{0<r'\leq r}\theta_A(x,r')<\delta$.\end{lemma}

\begin{proof} Let $\delta>0$ be given and fix parameters $\varepsilon>0$, $\sigma>0$, and $\tau>0$ to be chosen later. Assume that $A\subset\RR^n$ is a non-empty set which satisfies (\ref{emu1}) at some location $x\in A$ and initial scale $r>0$. Also assume that $\theta_A(x,r)<\tau$. Then by definition there exists a hyperplane $L\in G(n,n-1)$ such that \begin{equation}\label{otter1} \HD[A\cap B(x,r),(x+L)\cap B(x,r)]<\tau r.\end{equation} On the other hand, since $\Theta^{\mathcal{H}_d}_A(x,r)<\varepsilon$, there exists a harmonic polynomial $h:\RR^n\rightarrow\RR$ such that $h(0)=0$ and $1\leq \deg h\leq d$ for which \begin{equation}\label{otter2} \HD[A\cap B(x,r),(x+\Sigma_h)\cap B(x,r)]<\varepsilon r.\end{equation} Combining (\ref{otter1}) and (\ref{otter2}), we have \begin{equation}\label{otter3} \HD[\Sigma_h\cap B_{r},L\cap B_{r}]< (\tau + \varepsilon)r. \end{equation} Set $\delta_*=\min\left\{\delta_{n,1},\dots,\delta_{n,d}\right\}$, where $\delta_{n,1},\dots, \delta_{n,d}$ denote the constants from Theorem \ref{polythm}, and assign $s=\sigma\delta_*C_{n,d}^{-1}$ where $C_{n,d}$ also denotes the constant from Theorem \ref{polythm}. Assume $\tau+\varepsilon<\delta_*$. By (\ref{otter3}) and Theorem \ref{polythm}, $\theta_{\Sigma_h}(0,2sr)<2\sigma$. Hence there exists $P\in G(n,n-1)$ such that \begin{equation}\label{otter4} \HD[\Sigma_h\cap B_{2sr}, P\cap B_{2sr}]< 2\sigma(2sr)=4\sigma sr.\end{equation} We will use (\ref{otter2}) and (\ref{otter4}) to estimate $\HD[A\cap B(x,sr), (x+P)\cap B(x,sr)]$.

First suppose that $x'\in A\cap B(x,sr)$. By (\ref{otter2}), $\dist(x', (x+\Sigma_h)\cap B(x,r))< \varepsilon r$. Hence there exists $y\in\Sigma_h$ such that $|x'-x-y|<\varepsilon r$. We now specify that $\varepsilon\leq s\approx \sigma$. Then $y \in \Sigma_h \cap B_{sr+\varepsilon r}\subset \Sigma_h\cap B_{2sr}$. Hence by (\ref{otter4}), $\dist(y,P\cap B_{2sr})< 4\sigma sr$. Choose $p\in P\cap B_{2sr}$ such that $|y-p|<4\sigma sr$. In fact, since $y\in B_{sr+\varepsilon r}$, we know $p\in B_{sr+\varepsilon r+4\sigma sr}$. Since $P$ is a hyperplane through the origin, we can find a second point $p'\in P\cap B_{sr}$ such that $|p'-p|\leq\varepsilon r+4\sigma sr$. Thus $x+p'\in (x+P)\cap B(x,sr)$ and \begin{equation} |x'-x-p'|\leq |x'-x-y|+|y-p|+|p-p'| < 2\varepsilon r + 8\sigma sr. \end{equation} We conclude that \begin{equation}\label{otter5} \dist (x', (x+P)\cap B(x,sr))< 2\varepsilon r+8\sigma sr\quad\text{for all }x'\in A\cap B(x,sr).\end{equation}

Next suppose that $x+p\in (x+P)\cap B(x,sr)$. Since $P$ is a hyperplane, we can select a second point $x+p'\in (x+P)\cap B(x,sr-\varepsilon r-4\sigma sr)$ such that $|p'-p|\leq \varepsilon r+4\sigma sr$. By (\ref{otter4}) there exists $x+y\in (x+\Sigma_h)\cap B(x,2sr)$ such that $|p'-y|< 4\sigma sr$. In fact, since $p\in B_{sr-\varepsilon r-4\sigma sr}$, we get $y\in B_{sr-\varepsilon r}$. By (\ref{otter2}) there exists $x'\in A\cap B(x,r)$ with $|x+y-x'|<\varepsilon r$. But since $y\in B_{sr-\varepsilon r}$, we know $x'\in A\cap B(x,sr)$ and \begin{equation} |x+p -x'| \leq |p-p'|+|p'-y|+|x+y-x'|< 2\varepsilon r+8\sigma sr.\end{equation} Thus \begin{equation} \label{otter6} \dist(x+p, A\cap B(x,sr))< 2\varepsilon r+8\sigma sr\quad\text{for all }x+p\in (x+P)\cap B(x,sr)\end{equation} Having established (\ref{otter5}) and (\ref{otter6}), we conclude \begin{equation}\label{otter7}\theta_A(x,sr)< 2\varepsilon/s+8\sigma\quad\text{provided }\tau+\varepsilon<\delta_*\text{ and } \varepsilon\leq s.\end{equation}

We are ready to choose parameters. Set $\tau=\min(\delta,\delta_*)/2<1$, put $\sigma=\tau/16$ (forcing $s=\sigma \delta_{n,d}^* C_{n,d}^{-1}<1$) and assign $\varepsilon=s\tau/4$. Then $\tau+\varepsilon \leq \delta_*/2+ \delta_*/8<\delta_*$ and $\varepsilon\leq s$. Hence $\theta_A(x,sr)< \tau$ by (\ref{otter7}).
We have proved if $\Theta^{\mathcal{H}_d}_A(x,r)<\varepsilon$ and $\theta_A(x,r)<\tau$, then on a smaller scale $\theta_A(x,sr)<\tau$, as well. For emphasis, we remark again that $s=s(n,d,\delta)<1$.

To finish the lemma, we now suppose that $A\subset\RR^n$, $x\in A$ and $r>0$ satisfy (\ref{emu1}) and  \begin{equation}\theta_A(x,r)<\eta:=s\tau/4.\end{equation} Then, by Lemma \ref{thetaonehalf},  \begin{equation}\theta_A(x,tr)< 4\eta/t< 4\eta/s=\tau\quad\text{for all }s<t\leq 1.\end{equation} Since $\Theta^{\mathcal{H}_d}_A(x,tr)<\varepsilon$ (by (\ref{emu1})) and $\theta_A(x,tr)<\tau$ for all $s<t\leq 1$, the argument above implies $\theta_A(x,str)<\tau$ for all $s<t\leq 1$, or equivalently, \begin{equation} \theta_A(x,tr)<\tau\quad\text{for all }s^2<t\leq 1.\end{equation} By a simple inductive argument, we conclude that $\theta_A(x,tr)<\tau$ for all $0<t\leq 1$. Therefore, since $\tau<\delta/2$, $\sup_{0<r'<r}\theta_A(x,r')<\delta$, as desired.
\end{proof}

If a set $\Gamma$ is locally well-approximated by $\mathcal{H}_d$, then all blow-ups of $\Gamma$ are zero sets of harmonic polynomials of degree at most $d$. If, in addition, $\Gamma$ has the feature that all blow-ups of $\Gamma$ are zero sets of \emph{homogeneous} harmonic polynomials, then we can say more. This is the main result of this section.

\begin{theorem} \label{generalopenthm} Suppose that $\Gamma\subset\RR^n$ is locally well-approximated by $\mathcal{H}_d$ for some $d\geq 1$. If every blow-up of $\Gamma$ is the zero set of a homogeneous harmonic polynomial, then $\Gamma$ can be written as a disjoint union \begin{equation} \Gamma=\Gamma_1\cup \Gamma_s\end{equation} with the following properties: \begin{enumerate}
 \item Every blow-up of $\Gamma$ centered at $x\in\Gamma_1$ is a hyperplane.
 \item Every blow-up of $\Gamma$ centered at $x\in\Gamma_s$ is the zero set of a homogeneous harmonic polynomial of degree at least 2.
 \item The set of flat points $\Gamma_1$ is open in $\Gamma$.
 \item The set of flat points $\Gamma_1$ is locally Reifenberg flat with vanishing constant.
 \item The set of flat points $\Gamma_1$ has Hausdorff dimension $n-1$.
\end{enumerate}
\end{theorem}

\begin{proof} Assume that $\Gamma\subset\RR^n$ is locally well-approximated by $\mathcal{H}_d$ for some $d\geq 1$. Moreover, assume that every blow-up of $\Gamma$ is the zero set of a homogeneous harmonic polynomial. Let $\delta_{n,k}$ $(1\leq k\leq d)$ be the constants from Theorem \ref{polythm}; and let $\varepsilon=\varepsilon(n,d,\delta)$ and $\eta=\eta(n,d,\delta)$ be the constants from Lemma \ref{emu} which correspond to $\delta=\min(\delta_{n,1}, \dots, \delta_{n,d})/2$.  We can partition $\Gamma$ into two sets $\Gamma_1$ and $\Gamma_s$ as follows. Set \begin{equation} \Gamma_1=\left\{x\in\Gamma: \liminf_{r\downarrow 0} \theta_\Gamma(x,r)<\eta/8\right\}\quad\text{and}\quad\Gamma_s=\left\{x\in\Gamma:\liminf_{r\downarrow 0} \theta_\Gamma(x,r)\geq \eta/8\right\}.\end{equation} Then $\Gamma=\Gamma_1\cup\Gamma_s$ and $\Gamma_1\cap\Gamma_s=\emptyset$. Since $\theta_{\Gamma}(x,r_i)\rightarrow 0$ along some sequence $r_i\rightarrow 0$ whenever $\Gamma$ blows-up to a hyperplane, it is clear that every blow-up of $\Gamma$ centered at $x\in\Gamma_s$ must be the zero set of a polynomial of degree at least 2. It remains to show that every blow-up of $\Gamma$ centered at $x\in\Gamma_1$ is a hyperplane; the set $\Gamma_1$ is relatively open in $\Gamma$; and $\Gamma_1$ is locally Reifenberg flat with vanishing constant (thus $\dim_H\Gamma_1=n-1$).

Fix $x_0\in\Gamma_1$. Because $\Gamma$ is locally well-approximated by $\mathcal{H}_d$, there exists $r_0\in(0,1)$ such that $\Theta^{\mathcal{H}_d}_{\Gamma}(x,r)<\varepsilon$ for every $x\in\Gamma\cap B(x_0,1)$ and for all $r\in(0,r_0)$.
Since $x_0\in\Gamma_1$, $\liminf_{r\downarrow 0}\theta_{\Gamma}(x_0,r)<\eta/8$. Hence we can find $r_1\in(0,r_0/2)$ such that $\theta_{\Gamma}(x_0,2r_1)<\eta/8$. Thus $\theta_{\Gamma}(x,r_1)<\eta$ for every $x\in \Gamma\cap B(x_0,r_1)$, by Lemma \ref{thetaonehalf}. Therefore, by Lemma \ref{emu}, \begin{equation}\label{data1}\theta_{\Gamma}(x,r)<\delta\quad\text{for all }x\in\Gamma\cap B(x_0,r_1),\ r\in(0,r_1).\end{equation}
We claim that every blow-up of $\Gamma$ centered at $x\in \Gamma\cap B(x_0,r_1)$ is a hyperplane. Indeed fix $x\in\Gamma\cap B(x_0,r_1)$ and assume $B$ is a blow-up of $\Gamma$ centered at $x$. On one hand, by our assumption on blow-ups of $\Gamma$, there exists a homogeneous harmonic polynomial $h:\RR^n\rightarrow\RR$ with $1\leq \deg h\leq d$ such that $B=\Sigma_h$. Thus there exists a sequence $r_i\rightarrow 0$ such that \begin{equation} \lim_{i\rightarrow\infty} \HD\left[\frac{\Gamma-x}{r_i}\cap B_1,\Sigma_h\cap B_1\right]=0.\end{equation} On the other hand, by (\ref{data1}), $\theta_{\Sigma_h}(0,1)\leq \liminf \theta^1_{\Gamma}(x,r_i)\leq \delta\leq \delta_{n,k}/2$, where $k=\deg h$. By Theorem \ref{polythm}, $Dh(0)\neq 0$. But since $h$ is homogeneous, $Dh(0)\neq 0$ if and only if $k=\deg h=1$ if and only if $\Sigma_h$ is a hyperplane. We have shown that for every $x_0\in\Gamma_1$ there exists $r_1>0$ such that every blow-up of $\Gamma$ centered at $x\in B(x_0,r_1)$ is a hyperplane. In other words, every $x\in\Gamma_1$ is a flat point of $\Gamma$ and $\Gamma_1$ is open in $\Gamma$.

Next we will demonstrate that $\Gamma_1$ is locally Reifenberg flat with vanishing constant. Let $\tau>0$ and let $K\subset\Gamma_1$ compact be given. Let $\varepsilon=\varepsilon(n,d,\tau)>0$ and $\eta=\eta(n,d,\tau)>0$ be the constants from Lemma \ref{emu}. Since $\Gamma_1$ is open and $K$ is compact, we can find $s_0>0$ such that $\Gamma\cap B(x,s)=\Gamma_1\cap B(x,s)$ for all $x\in K$ and for all $s\in(0,s_0)$. And since $\Gamma$ is locally well-approximated by $\mathcal{H}_d$, there exists $s_1\in (0,s_0)$ such that $\Theta^{\mathcal{H}_d}_{\Gamma}(x,s)<\varepsilon$ for every $x\in K$ and $s\in (0,s_1)$. Now for each $x\in K$ there exists $s_x\in(0,s_1)$ with $\theta_{\Gamma}(x,2s_x)<\eta/8$ (since $x\in K\subset\Gamma_1$ is a flat point of $\Gamma$). Hence, by Lemma \ref{thetaonehalf}, $\theta_{\Gamma}(x',s_x)<\eta$ for all $x'\in \Gamma\cap B(x,s_x)$, for all $x\in K$. Thus, by Lemma \ref{emu}, \begin{equation} \label{wardy} \theta_{\Gamma}(x',s')<\tau\quad\text{for all }x'\in \Gamma\cap B(x,s_x),\text{ for all }s'\in(0,s_x),\text{ for all }x\in K.\end{equation} But $K$ is compact, so $K$ admits a finite cover of the form $\{B(x_i,s_{x_i}):x_1,\dots,x_m\in K\}$. Letting $s_*=\min\{s_{x_1},\dots,s_{x_m}\}$, we conclude \begin{equation} \theta_{\Gamma_1}(x,s')=\theta_{\Gamma}(x,s')<\tau\quad\text{for all }x\in K,\text{ for all }s'\in(0,s_*).\end{equation} Thus, since $K\subset\Gamma_1$ was an arbitrary compact subset, $\Gamma_1$ is locally $\tau$-Reifenberg flat. Therefore, since $\tau>0$ was arbitrary, $\Gamma_1$ is locally Reifenberg flat with vanishing constant, as desired.

Finally, by Corollary \ref{vanishdim}, any locally Reifenberg flat set with vanishing constant, and in particular $\Gamma_1$, has Hausdorff dimension $n-1$.\end{proof}

\section{Free Boundary Regularity for Harmonic Measure from Two Sides}
\label{SectFreeBoundary}

The goal of this section is to document the structure of the free boundary for harmonic measure from two sides under weak regularity (Theorem \ref{structthm}). In order to state the result, we must remind the reader of several standard definitions from harmonic analysis (\S6.1). The statement and proof of the structure theorem are then given in \S6.2 and \S6.3. For an introduction to free boundary regularity problems for harmonic measure, we recommend the reader to the book \cite{CKL} by Capogna, Kenig and Lanzani. For recent generalizations to free boundary problems for $p$-harmonic measure, see Lewis and Nystr\"om \cite{LN}.

\subsection{Definitions and Conventions}

In \cite{JK} Jerison and Kenig introduced NTA domains---a natural class of domains on which Fatou type convergence theorems hold for harmonic functions. In the plane, a bounded simply connected domain $\Omega\subset\RR^2$ is an NTA domain if and only if $\Omega$ is a quasidisk (the image of the unit disk under a global quasiconformal mapping of the plane). In higher dimensions, while every quasiball (the image of the unit ball under a global quasiconformal mapping of space) in $\RR^n$, $n\geq 3$, is also a bounded NTA domain, there exist bounded NTA domains homeomorphic to a ball in $\RR^n$ which are not quasiballs. The reader may consult \cite{JK} for more information. Also see \cite{KT97} where Kenig and Toro demonstrate that a domain in $\RR^n$ whose boundary is $\delta$-Reifenberg flat is an NTA domain provided that $\delta<\delta_n$ is sufficiently small.

The definition of NTA domains is based upon two geometric conditions, which are quantitative, scale-invariant versions of openness and path connectedness.

\begin{definition}An open set $\Omega\subset\RR^n$ satisfies the \emph{corkscrew condition} with constants $M>1$ and $R>0$ provided that for every $x\in\partial\Omega$ and $0<r<R$ there exists a \emph{non-tangential point} $A=A(x,r)\in\Omega\cap B(x,r)$ such that $\dist(A,\partial\Omega)> M^{-1}r$.\end{definition}

For $X_1,X_2\in\Omega$ a \emph{Harnack chain} from $X_1$ to $X_2$ is a sequence of closed balls inside $\Omega$ such that the first ball contains $X_1$, the last contains $X_2$, and consecutive balls intersect. The \emph{length} of a Harnack chain is the number of balls in the chain.

\begin{definition}A domain $\Omega\subset\RR^n$ satisfies the \emph{Harnack chain condition} with constants $M>1$ and $R>0$ if for every $x\in\partial\Omega$ and $0<r<R$ when $X_1,X_2\in \Omega\cap B(x,r)$ satisfy \begin{equation} \min_{j=1,2}\dist(X_j,\partial\Omega)>\varepsilon\quad\text{and}\quad|X_1-X_2|< 2^k\varepsilon\end{equation} then there is a Harnack chain from $X_1$ to $X_2$ of length $Mk$ such that the diameter of each ball is bounded below by $M^{-1}\min_{j=1,2}\dist(X_j,\partial\Omega)$.\end{definition}

\begin{definition}\label{NTAdefn}A domain $\Omega\subset\RR^n$ is \emph{non-tangentially accessible} or \emph{NTA} if there exist $M>1$ and $R>0$ such that (i) $\Omega$ satisfies the corkscrew and Harnack chain conditions, (ii) $\RR^n\setminus\overline{\Omega}$ satisfies the corkscrew condition. If $\partial\Omega$ is unbounded then we require $R=\infty$.\end{definition}

The exterior corkscrew condition guarantees that an NTA domain is \emph{regular} for the Dirichlet problem. Thus for every choice $f\in C_c(\partial\Omega)$ of continuous boundary data with compact support there exists a unique function $u\in C(\overline{\Omega})\cap C^2(\Omega)$ such that $u$ is harmonic in $\Omega$ and $u=f$ on $\partial\Omega$. On a regular domain $\Omega$, \emph{harmonic measure} $\omega^X$ with \emph{pole} at $X\in\Omega$ is the unique probability measure supported on $\partial\Omega$ such that \begin{equation*} u(X)=\int_{\partial\Omega} f(Q)\,d\omega^X(Q)\quad\text{for all }f\in C_c(\partial\Omega).\end{equation*} Since $\omega^{X}\ll\omega^{Y}\ll\omega^{X}$ for different choices of pole $X,Y\in\Omega$ (by Harnack's inequality for positive harmonic functions), we may drop the pole from our notation and refer to the harmonic measure $\omega$ ($=\omega^{X_0}$) of $\Omega$ (with respect to some fixed, unspecified pole $X_0\in\Omega$).

A nice feature of harmonic measure on NTA domains is that the harmonic measure $\omega$ is locally doubling; see Lemmas 4.8 and 4.11 in \cite{JK}. In particular, this property implies that $\omega(\partial\Omega\cap B(x,r))>0$ for every location $x\in\partial\Omega$ and every scale $r>0$.

On unbounded NTA domains there is a related notion of harmonic measure with pole at infinity, whose existence is guaranteed by the following lemma. \begin{lemma}[\cite{KT99} Lemma 3.7, Corollary 3.2]\label{omegaInfiniteDoubling} Let $\Omega\subset\RR^n$ be an unbounded NTA domain. There exists a doubling Radon measure $\omega^\infty$ supported on $\partial\Omega$ satisfying \begin{equation} \int_{\partial\Omega} \varphi\, d\omega^\infty = \int_\Omega u\Delta\varphi\quad\text{for all }\varphi\in C_c^\infty(\RR^n)\end{equation} where \begin{equation} \left\{\begin{array}{rl} \Delta u=0 &\text{in }\Omega\\ u>0 &\text{in }\Omega\\ u=0 &\text{on }\partial\Omega.\end{array}\right.\end{equation} The measure $\omega^\infty$ and Green function $u$ are unique up to multiplication by a positive scalar. We call $\omega^\infty$ a harmonic measure of $\Omega$ with pole at infinity.\end{lemma}

Below we only work with domains whose interior and exterior are both NTA domains. Note that as a consequence of the corkscrew conditions, the interior $\Omega^+$ and exterior $\Omega^-$ of a 2-sided NTA domain $\Omega$ have a common boundary: $\partial\Omega^+=\partial\Omega=\partial\Omega^-$.

\begin{definition}A domain $\Omega\subset\RR^n$ is \emph{2-sided NTA} if $\Omega^+=\Omega$ and $\Omega^-=\RR^n\setminus\overline{\Omega}$ are both NTA with the same constants; i.e.\ there exists $M>1$ and $R>0$  such that $\Omega^\pm$ satisfy the corkscrew and Harnack chain conditions. When $\partial\Omega$ is unbounded, we require  $R=\infty$.\end{definition}

We also need the following classes of functions.

\begin{definition}Let $\Omega\subset\RR^n$ be a NTA domain. We say that $f\in L^2_{\mathrm{loc}}(d\omega)$ has \emph{bounded mean oscillation} with respect to the harmonic measure $\omega$ and write $f\in\BMO(d\omega)$ if \begin{equation}\sup_{r>0}\sup_{Q\in\partial\Omega}\left(\omega(B(Q,r))^{-1}\int_{B(Q,r)} |f-f_{Q,r}|^2d\omega\right)^{1/2}<\infty\end{equation} where $f_{Q,r}=\omega(B(Q,r))^{-1}\int_{B(Q,r)}fd\omega$ denotes the average of $f$ over the ball.\end{definition}

\begin{definition}\label{defnVMO}Let $\Omega\subset\RR^n$ be a NTA domain. Let $\VMO(d\omega)$ denote the closure of the set of bounded uniformly continuous functions on $\partial\Omega$ in $\BMO(d\omega)$. If $f\in\VMO(d\omega)$, then we say $f$ has \emph{vanishing mean oscillation} with respect to the harmonic measure $\omega$.\end{definition}

\subsection{Structure Theorem}

The following statement incorporates and extends results which first appeared in Kenig and Toro \cite{KT06} and Badger \cite{Badger1}. If $\Omega^+$ and $\Omega^-$ are unbounded, then we allow $\omega^+$ (harmonic measure on $\Omega^+$) and $\omega^-$ (harmonic measure on $\Omega^-$) to have finite poles or poles at infinity; otherwise we require that $\omega^+$ and $\omega^-$ have finite poles. The new aspect of the theorem presented here is statement (ii) about the flat points $\Gamma_1$.

\begin{theorem}\label{structthm} Assume that $\Omega\subset\RR^n$ is a 2-sided NTA domain. If $\omega^+\ll\omega^-\ll\omega^+$ and $\log \frac{d\omega^-}{d\omega^+}\in\VMO(d\omega^+)$, then there is $d_0>0$ such that $\partial\Omega$ is locally well-approximated by $\mathcal{H}_{d_0}$. Moreover, $\partial\Omega$ can be partitioned into sets $\Gamma_d$ ($1\leq d\leq d_0$), \begin{equation} \partial\Omega=\Gamma_1\cup\dots\cup\Gamma_{d_0},\quad i\neq j \Rightarrow \Gamma_i\cap\Gamma_j=\emptyset,\end{equation} with the following properties: \begin{enumerate}
\item Every blow-up of $\partial\Omega$ centered a point $x\in\Gamma_d$ is the zero set of a homogeneous harmonic polynomial of degree $d$ which separates $\RR^n$ into two components.
\item The set of flat points $\Gamma_1$ is open and dense in $\partial\Omega$; $\Gamma_1$ is locally Reifenberg flat with vanishing constant; and $\Gamma_1$ has Hausdorff dimension $n-1$.
\item The set of ``singularities" $\partial\Omega\setminus\Gamma_1=\Gamma_2\cup\dots\cup\Gamma_{d_0}$ is closed and has harmonic measure zero: $\omega^\pm(\partial\Omega\setminus\Gamma_1)=0$.\end{enumerate}
\end{theorem}

\begin{remark}In Theorem \ref{structthm} the phrase ``separates $\RR^n$ into two components'' means that if the zero set $\Sigma_h$ of a homogeneous harmonic polynomial $h$ is a blow-up of $\partial\Omega$ then the open set $\RR^n\setminus\Sigma_h$ has exactly two connected components. The existence of polynomials with this separation property depends on the dimension $n$. When $n=2$, the zero set $\Sigma_h$ of a homogeneous harmonic polynomial separates $\RR^2$ into two components if and only if $\deg h=1$. When $n=3$, Lewy \cite{Lewy} showed that $\Sigma_h$ can separate $\RR^3$ into two components only if $\deg h$ is odd. Thus the separation condition on $\Sigma_h$ restricts the existence of the sets $\Gamma_k$  in low dimensions.\end{remark}

\begin{corollary} If $\Omega\subset\RR^2$ satisfies the hypothesis of the Theorem \ref{structthm}, then $\partial\Omega=\Gamma_1$.\end{corollary}

\begin{corollary} If $\Omega\subset\RR^3$ satisfies the hypothesis of the Theorem \ref{structthm}, then $d_0\geq 1$ is an odd integer and $\partial\Omega=\Gamma_1\cup\Gamma_3\cup\dots\cup\Gamma_{d_0}$.
\end{corollary}

\begin{example}\label{badexample} Consider the homogeneous harmonic polynomial $h:\RR^n\rightarrow \RR$ ($n\geq 3$), \begin{equation}h(X)=X_1^2(X_2-X_3)+X_2^2(X_3-X_1)+X_3^2(X_1-X_2)-X_1X_2X_3.\end{equation} Then the
domain $\Omega=\{X\in\RR^n:h(X)>0\}$ is a 2-sided NTA domain; in particular, $\partial\Omega=\{X\in\RR^n:h(X)=0\}=\Sigma_h$ separates $\RR^n$ into two components (see Figure 1.1). Let $\omega^+$ and $\omega^-$ denote harmonic measures of $\Omega^+$ and $\Omega^-$ with pole at infinity. Then there exists a constant $c>0$ such that $\omega^+=c\omega^-$, $\log f_\Omega\equiv 0$ and \begin{equation} \frac{\partial\Omega-X}{r}=\Sigma_h\quad\text{for all } X=(0,0,0,X_4,\dots,X_n)\text{ and }r>0.\end{equation} In particular, $\Sigma_h$ is a blow-up of $\partial\Omega$ at the origin. Thus $0\in\Gamma_3$ and non-planar blow-ups of the boundary can appear even when $\log f_\Omega$ is real-analytic. Furthermore, for all $n\geq 3$, this example shows that it is possible for the set of ``singularities''  $\partial\Omega\setminus\Gamma_1=\Gamma_2\cup\dots\cup\Gamma_{d_0}$ to have Hausdorff dimension $\geq n-3$.
\end{example}

\begin{remark} To our knowledge, the first explicit example of a non-planar zero set of a harmonic polynomial dividing space into two components was given by Szulkin \cite{Sz}.\end{remark}

An obvious modification of the domain $\Omega$ in Example \ref{badexample} shows that:

\begin{proposition} The zero set $\Sigma_h$ of a harmonic polynomial $h:\RR^n\rightarrow\RR$ appears as a blow-up of $\partial\Omega$ for some $\Omega\subset\RR^n$ satisfying the hypothesis of Theorem \ref{structthm} if and only if $h$ is homogeneous and $\Sigma_h$ separates $\RR^n$ into two components.\end{proposition}

\begin{proof} Necessity was established by the Theorem \ref{structthm}, so it remains to check sufficiency. Let $h:\RR^n\rightarrow\RR$ be a homogeneous harmonic polynomial such that $\Sigma_h$ separates $\RR^n$ into two components. Then $\Omega=\{X\in\RR^n:h(X)>0\}$ is a 2-sided NTA domain which satisfies the hypothesis of Theorem \ref{structthm}. Moreover, since \begin{equation}\frac{\partial\Omega}{r}=\Sigma_h\quad\text{for all }r>0,\end{equation} $\Sigma_h$ is the unique blow-up of $\partial\Omega=\Sigma_h$ at the origin.\end{proof}

Several questions about the sets in the decomposition in Theorem \ref{structthm} remain open. The first pair of questions involve the singularities in the boundary.

\begin{problem}\label{sprob1} Is $\Gamma_d$ a closed set for each $d\geq 2$? \end{problem}

\begin{problem}\label{sprob2} Find a sharp upper bound on the Hausdorff dimension of $\partial\Omega\setminus\Gamma_1$.\end{problem}

\begin{remark}We expect that the resolution of Problem \ref{sprob1} is tied to the question posed in Remark \ref{sharmseparate}. And based on Example \ref{badexample}, we conjecture that $\dim_H \partial\Omega\setminus\Gamma_1\leq n-3$. (One might predict $\dim_H \partial\Omega\setminus\Gamma_1\leq n-2$, but we believe the requirement that $\Sigma_h$ separate $\RR^n$ into two components forces the smaller upper bound.)\end{remark}

The next problem is related to a conjecture of Bishop in \cite{Bishop} about the rectifiability of harmonic measure in dimensions $n\geq 3$.

\begin{problem}\label{sprob3} Is it always possible to decompose $\Gamma_1$ as $\Gamma_1=\Gamma_1^{\Good}\cup\Gamma_1^{\Null}$ so that $\Gamma_1^{\Good}$ is an $(n-1)$-rectifiable set and $\omega^\pm\ll \mathcal{H}^{n-1}\res \Gamma_1^{\Good}$, and so that $\omega^\pm(\Gamma_1^{\Null})=0$?
\end{problem}

\begin{remark} The answer to Problem \ref{sprob3} is yes under the additional assumption that $\mathcal{H}^{n-1}\res\partial\Omega$ is a Radon measure (e.g.\ if $\mathcal{H}^{n-1}(\partial\Omega)<\infty$). For instance, one can verify this assertion by combining a recent result of Kenig, Preiss and Toro \cite{KPT} (see Corollary 4.2) with a recent result of Badger \cite{Badger2} (see Theorem 1.2).\end{remark}

A final set of open problems concern the question of higher regularity. If one assumes extra regularity on the logarithm of the two-sided kernel $f=d\omega^-/d\omega^+$ beyond $\VMO$, then do the flat points $\Gamma_1$ have extra regularity beyond being locally Reifenberg flat with vanishing constant? For example,

\begin{problem} \label{sprob4} If $\log \frac{d\omega^-}{d\omega^+}\in C^\infty(\partial\Omega)$, then is $\Gamma_1$ locally the $C^\infty$ image of a hyperplane? \end{problem}

One can ask a similar (if more difficult) question at the singularities. For example,

\begin{problem} \label{sprob5} If $\log\frac{d\omega^-}{d\omega^+}\in C^\infty(\partial\Omega)$ and $x\in\Gamma_d$ ($d\geq 2$), then is $\partial\Omega$ near $x$ locally the $C^\infty$ image of some homogeneous harmonic polynomial of degree $d$ which separates space into two components?
\end{problem}

\begin{remark} A resolution of the parametrization problem posed in Problem \ref{sprob5} likely depends on the answer to the question in Remark \ref{sharmseparate}. \end{remark}

\subsection{Proof of Theorem \ref{structthm}}

The structure theorem (Theorem \ref{structthm}) is an amalgamation of Theorems 4.2 and 4.4 in \cite{KT06}, Theorem 1.3 in \cite{Badger1} and Theorem \ref{generalopenthm} above.

Assume that $\Omega\subset\RR^n$ is a 2-sided NTA domain such that $\omega^+\ll\omega^-\ll\omega^+$ and $\log \frac{d\omega^-}{d\omega^+}\in\VMO(d\omega^+)$. The first statement that we need to verify is that there exists an integer $d_0\geq 1$ such that $\partial\Omega$ is locally well-approximated by $\mathcal{H}_{d_0}$ (recall Definitions \ref{lapprox} and \ref{HdApprox}). By Theorems 4.2 and 4.4 in \cite{KT06}: there exists $d_0>0$ (depending only on $n$ and the NTA constants of $\Omega^+$ and $\Omega^-$) such that if $x\in\partial\Omega$, if $x_i\in\partial\Omega$ is a sequence such that $x_i\rightarrow x$, and if $r_i\rightarrow 0$ is a vanishing sequence of positive numbers, then there exists a subsequence $(x_{ij},r_{ij})_{j\geq 1}$ of $(x_i,r_i)_{i\geq 1}$ and a nonconstant harmonic polynomial $h:\RR^n\rightarrow\RR$ of degree at most $d_0$ such that $r_{ij}^{-1}(\partial\Omega-x_{ij})$ converges to $\Sigma_h$ in the Hausdorff distance, uniformly on compact sets. (Said more briefly, all $\emph{pseudo blow-ups}$ of $\partial\Omega$ are zero sets of harmonic polynomials of degree at most $d_0$.) Now suppose for contradiction that there exists a compact set $K\subset\partial\Omega$ such that $\Theta^{\mathcal{H}_{d_0}}_{\partial\Omega}(x,r)$ does not vanish uniformly on $K$ as $r\rightarrow 0$. Then there exist $\varepsilon>0$ and sequences $x_i\in K$ and $r_i\downarrow 0$ so that \begin{equation}\label{goat1}\Theta^{\mathcal{H}_{d_0}}_{\partial\Omega}(x_i,r_i)\geq \varepsilon\quad\text{for all }i\geq 1.\end{equation} Passing to a subsequence, we may assume that $x_i\rightarrow x\in K$ (since $K$ is compact). Then Theorems 4.2 and 4.4 in \cite{KT06} yield a further subsequence $(x_{ij},r_{ij})_{j=1}^\infty$ of $(x_i,r_i)_{i=1}^\infty$ such that $\lim_{j\rightarrow\infty}\Theta^{\mathcal{H}_{d_0}}_{\partial\Omega}(x_{ij},r_{ij})=0$. This contradicts (\ref{goat1}). Therefore, our supposition was false, and hence, we get that $\lim_{r\downarrow 0}\Theta^{\mathcal{H}_{d_0}}_{\partial\Omega}(x,r)=0$ uniformly on $K$. In other words, $\partial\Omega$ is locally well-approximated by $\mathcal{H}_{d_0}$, as desired.

Next, by Theorem 1.3 in \cite{Badger1}, we can decompose $\partial\Omega$ into disjoint sets $\Gamma_d$ ($1\leq d\leq d_0$), \begin{equation} \partial\Omega=\Gamma_1\cup\dots\cup\Gamma_{d_0},\quad i\neq j \Rightarrow \Gamma_i\cap\Gamma_j=\emptyset,\end{equation} where: \begin{itemize} \item Every blow-up of $\partial\Omega$ centered at $x\in\Gamma_d$ is the zero set of a \emph{homogeneous} harmonic polynomial of degree $d$, such that the zero set divides space into two components.
\item The set of ``singularities" $\partial\Omega\setminus\Gamma_1$ has zero harmonic measure: $\omega^\pm(\partial\Omega\setminus\Gamma_1)=0$.
\end{itemize} Now because $\omega^\pm(\partial\Omega\cap B(x,r))>0$ for all $x\in\partial\Omega$ and $r>0$, and because $\omega^\pm(\partial\Omega\setminus\Gamma_1)=0$, we conclude that $\omega^\pm(\Gamma_1\cap B(x,r))>0$ for all $x\in\partial\Omega$ and $r>0$. In particular, this implies $\Gamma_1$ is dense in $\partial\Omega$.

Finally, because $\partial\Omega$ is locally well-approximated by $\mathcal{H}_{d_0}$ and every blow-up of $\partial\Omega$ is the zero set of a homogeneous harmonic polynomial, by Theorem \ref{generalopenthm} above, we can also decompose $\partial\Omega$ as \begin{equation}\partial\Omega=\Gamma_1\cup\Gamma_s\end{equation} where $\Gamma_1$ is the set of flat points of $\partial\Omega$ and blow-ups of $\partial\Omega$ centered at $x\in\Gamma_s$ are zero sets of homogeneous harmonic polynomials of degree at least 2. In particular $\Gamma_s=\partial\Omega\setminus\Gamma_1=\Gamma_2\cup\dots\cup\Gamma_{d_0}$. By Theorem \ref{generalopenthm}, $\Gamma_1$ is open, locally Reifenberg flat with vanishing constant, and $\Gamma_1$ has Hausdorff dimension $n-1$. This completes the proof of Theorem \ref{structthm}.

\begin{ack}\label{ackref}
The research in this article first appeared in the author's Ph.D.\ thesis at the University of Washington, under the supervision of Professor Tatiana Toro. The author would like to thank his advisor for introducing him to free boundary problems for harmonic measure, for her continual encouragement, and for her generosity in time and ideas. The author also would like to thank Professor Guy David at Universit\'e Paris--Sud XI for useful conversations, which helped lead the author to Theorem \ref{polythm} during the author's visit to Orsay in Spring 2010.
\end{ack}

\end{document}